\documentclass{amsart}

\usepackage{amsmath,amsfonts,amssymb,enumerate,tikz}
\usepackage[all, cmtip]{xy}

\theoremstyle{definition}
\newtheorem{definition}{Definition}
\newtheorem{example}{Example}
\newtheorem{remark}{Remark}

\theoremstyle{plain}

\newtheorem{theorem}{Theorem}
\newtheorem{proposition}{Proposition}
\newtheorem{corollary}{Corollary}

\newcommand{\C}{\mathbb{C}}
\newcommand{\Proj}{\mathbb{P}}
\newcommand{\Z}{\mathbb{Z}}
\newcommand{\R}{\mathbb{R}}
\newcommand{\calO}{\mathcal{O}}

\newcommand{\calF}{\mathcal{F}}
\newcommand{\calE}{\mathcal{E}}
\newcommand{\calS}{\mathcal{S}}
\newcommand{\calC}{\mathcal{C}}

\newcommand{\Id}{\mathrm{Id}}
\DeclareMathOperator{\NS}{NS}
\DeclareMathOperator{\T}{T}
\DeclareMathOperator{\Pic}{Pic}
\DeclareMathOperator{\rank}{rank}
\DeclareMathOperator{\Image}{Im}
\DeclareMathOperator{\MW}{MW}
\DeclareMathOperator{\Aut}{Aut}

\DeclareMathOperator{\Amp}{Amp}
\DeclareMathOperator{\K}{K}
\DeclareMathOperator{\disc}{disc}

\begin{document}

\title{The Geometry and Moduli of {K3} Surfaces}

\author{Andrew Harder}
\address{Department of Mathematical and Statistical Sciences, 632 CAB, University of Alberta, Edmonton, Alberta, T6G 2G1, Canada}
\email{aharder@ualberta.ca} 
\thanks{A. Harder was supported by an NSERC PGS D scholarship and a University of Alberta Doctoral
Recruitment Scholarship.}

\author{Alan Thompson}
\address{Department of Mathematical and Statistical Sciences, 632 CAB, University of Alberta, Edmonton, Alberta, T6G 2G1, Canada}
\email{amthomps@ualberta.ca}
\thanks{A. Thompson was supported by a Fields-Ontario-PIMS postdoctoral fellowship with funding provided by NSERC, the Ontario Ministry of Training, Colleges and Universities, and an Alberta Advanced Education and Technology Grant.}

\begin{abstract} These notes will give an introduction to the theory of K3 surfaces. We begin with some general  results on K3 surfaces, including the construction of their moduli space and some of its properties. We then move on to focus on the theory of polarized K3 surfaces, studying their moduli, degenerations and the compactification problem. This theory is then further enhanced to a discussion of lattice polarized K3 surfaces, which provide a rich source of explicit examples, including a large class of lattice polarizations coming from elliptic fibrations. Finally, we conclude by discussing the ample and K\"ahler cones of K3 surfaces, and give some of their applications.\end{abstract}

\maketitle

\section{General Results on K3 Surfaces}

We begin by recalling the definition of a K3 surface.

\begin{definition} A \emph{K3 surface} $S$ is a smooth compact complex surface with trivial canonical bundle $\omega_S \cong \calO_S$ and $h^1(S,\calO_S) = 0$.
\end{definition}

\begin{remark} Note that an arbitrary K3 surface $S$ is not necessarily projective, but every K3 surface is  K\"{a}hler. This was first proved by Siu \cite{ek3sk} who, by treating the K3 case, completed the proof of a conjecture of Kodaira \cite[Sect. XII.1]{cm} stating that every smooth compact complex surface with even first Betti number is K\"{a}hler. A direct proof of this conjecture may be found in \cite[Thm. IV.3.1]{bpv}.
\end{remark}

Unless otherwise stated, throughout these notes $S$ will denote an arbitrary K3 surface. In the remainder of this section we will study the geometry of $S$, then use this to initiate our study of the moduli space of K3 surfaces. Our main reference for this section will be \cite[Chap. VIII]{bpv}.

\subsection{Hodge Theory}\label{section:hodge}

We begin by studying the Hodge theory of a K3 surface $S$. The Hodge diamond of $S$ has the form
\[\begin{array}{ccccc} && h^{0,0} &&\\ &h^{1,0} & & h^{0,1} & \\ h^{2,0} & & h^{1,1} & & h^{0,2} \\ & h^{2,1} & & h^{1,2} & \\ && h^{2,2} &&  \end{array}\quad =\quad \begin{array}{ccccc} && 1 &&\\ &0 & & 0 & \\ 1 & & 20 & & 1 \\ & 0 & & 0 & \\ && 1 &&  \end{array}.\]
We note that this is largely trivial: the only interesting behaviour happens in the second cohomology group. As we shall see, the structure of this cohomology group determines the isomorphism class of a K3 surface, so can be used to construct a moduli space for K3 surfaces. 

The second cohomology group $H^2(S,\Z)$ with the cup-product pairing $\langle \cdot , \cdot\rangle$ forms a lattice isometric to the \emph{K3 lattice}
\[ \Lambda_{\mathrm{K}3} := H \oplus H \oplus H \oplus (-E_8) \oplus (-E_8),\]
where $H$ is the hyperbolic plane (an even, unimodular, indefinite lattice of rank $2$) and $E_8$ is the even, unimodular, positive definite lattice of rank $8$ corresponding to the Dynkin diagram $E_8$. The lattice $\Lambda_{\mathrm{K}3}$ is a non-degenerate even lattice of rank $22$ and signature $(3,19)$ (for the reader unfamiliar with lattice theory, we have included a short appendix containing results and definitions relevant to these notes).

There are two important sublattices of $H^2(S,\Z)$ that appear frequently in the study of K3 surfaces. The first is the \emph{N\'{e}ron-Severi lattice} $\NS(S)$, given by
\[\NS(S) := H^{1,1}(S) \cap H^2(S,\Z) \]
(here we identify $H^2(S,\Z)$ with its image under the natural embedding $H^2(S,\Z) \hookrightarrow H^2(S,\C)$). By the Lefschetz theorem on $(1,1)$-classes \cite[Thm. IV.2.13]{bpv}, $\NS(S)$ is isomorphic to the Picard lattice $\mathrm{Pic}(S)$, with isomorphism induced by the first Chern class map. 

The second important sublattice of $H^2(S,\Z)$ is the \emph{transcendental lattice} $\T(S)$. It is defined to be the smallest sublattice  of $H^2(S,\Z)$ whose complexification contains a generator $\sigma$ of $H^{2,0}(S)$. In the case where $\NS(S)$ is nondegenerate (which happens, for instance, when $S$ is projective), then the transcendental lattice is equal to the orthogonal complement of $\NS(S)$ in $H^2(S,\Z)$.

The structure of the second cohomology of $S$ is an important object to study, as it determines the isomorphism class of $S$.

\begin{theorem}[Weak Torelli] \textup{\cite[Cor. VIII.11.2]{bpv}} \label{weaktorelli} Two K3 surfaces $S$ and $S'$ are isomorphic if and only if there is a lattice isometry $H^2(S,\Z) \to H^2(S',\Z)$, whose $\C$-linear extension $H^2(S,\C) \to H^2(S',\C)$ preserves the Hodge decomposition \textup{(}such an isometry is called a \emph{Hodge isometry}\textup{)}.
\end{theorem}

\subsection{The Period Mapping}

We can use the weak Torelli theorem to begin constructing a moduli space for K3 surfaces. We start by defining a \emph{marking} on the K3 surface $S$.

\begin{definition} \label{defn:marked} A \emph{marking} on $S$ is a choice of isometry $\phi\colon H^2(S,\Z) \to \Lambda_{\mathrm{K}3}$. We say that $(S,\phi)$ is a \emph{marked K3 surface}.
\end{definition}

Since the canonical bundle of $S$ is trivial, we have $H^{2,0}(S) := H^0(S,\Omega^2_S) = H^0(S,\calO_S)$. Let $\sigma \in H^{2,0}(S)$ be any nonzero element. Then $\sigma$ is a nowhere vanishing $2$-form on $S$. Using the Hodge decomposition, we may treat $\sigma$ as an element of $H^2(S,\C)$. This cohomology group carries a bilinear form $\langle \cdot,\cdot\rangle$, given by the $\C$-linear extension of the cup-product pairing, with respect to which we have $\langle\sigma,\sigma\rangle = 0$ and $\langle\sigma,\overline{\sigma}\rangle >0$. 

If $\phi$ is a marking for $S$ and $\phi_{\C}\colon H^2(S,\C) \to \Lambda_{\mathrm{K}3} \otimes \C$ is its $\C$-linear extension, then $\phi_{\C}(H^{2,0}(S))$ is a line through the origin in $\Lambda_{\mathrm{K}3} \otimes \C$ spanned by $\phi_{\C}(\sigma)$. Projectivising, we see that $\phi_{\C}(H^{2,0}(S))$ defines a point in
\[\Omega_{\mathrm{K}3} := \{[\sigma]\in \Proj(\Lambda_{\mathrm{K}3} \otimes \C) \mid \langle\sigma,\sigma\rangle = 0, \langle\sigma,\ \overline{\sigma}\rangle > 0\}.\]

$\Omega_{\mathrm{K}3}$ is a $20$-dimensional complex manifold called the \emph{period space of K3 surfaces}. The point defined by $\phi_{\C}(H^{2,0}(S))$ is the \emph{period point of the marked K3 surface $(S,\phi)$}.

The Weak Torelli theorem (Thm. \ref{weaktorelli}) gives that two K3 surfaces are isomorphic if and only if there are markings for them such that the corresponding period points are the same. 

Now we extend this idea to families. Let $\pi \colon \calS \to U$ be a flat family of K3 surfaces over a small contractible open set $U$ and let $S$ be a fibre of $\pi$. A choice of marking $\phi\colon H^2(S,\Z) \to \Lambda_{\mathrm{K3}}$ for $S$ can be extended uniquely to a marking $\phi_U\colon R^2\pi_*\Z \to (\Lambda_{\mathrm{K3}})_U$ for the family $\calS$, where $(\Lambda_{\mathrm{K3}})_U$ denotes the constant sheaf with fibre $\Lambda_{\mathrm{K3}}$ on $U$. Applying the above construction to the marked K3 surfaces in the family $\calS$, we obtain a holomorphic map $U \to \Omega_{\mathrm{K3}}$, called the \emph{period mapping} associated to the family $\pi\colon \calS \to U$. 

Applying this to the case where $\pi\colon \calS \to U$ is a representative of the versal deformation of $S$, one finds:

\begin{theorem}[Local Torelli] \textup{\cite[Thm. VIII.7.3]{bpv}} \label{localtorelli} For any marked K3 surface $S$, the period mapping from the versal deformation space of $S$ to $\Omega_{\mathrm{K}3}$ is a local isomorphism.\end{theorem}

This shows that the period mapping is well-behaved under small deformations of a marked K3 surface. Moreover, we have:

\begin{theorem}[Surjectivity of the Period Map] \textup{\cite[Cor. VIII.14.2]{bpv}} \label{surjectivity}  Every point of $\Omega_{\mathrm{K}3}$ occurs as the period point of some marked K3 surface.\end{theorem}

Putting these elements together, we seem to be close to constructing a (coarse) moduli space for K3 surfaces: we have a space $\Omega_{\mathrm{K3}}$ whose points correspond to marked K3 surfaces, and any family of marked K3 surfaces $\pi\colon \calS \to U$ gives rise to a map $U \to \Omega_{\mathrm{K3}}$. All that remains is to quotient $\Omega_{\mathrm{K}3}$ by the action of the group $\Gamma$ of isometries of $\Lambda_{\mathrm{K}3}$ to identify period points corresponding to different markings on the same K3 surface. However, on closer inspection one finds that this group action is not properly discontinuous, so the quotient will have undesirable properties: in particular, it won't be Hausdorff. More details may be found in \cite[Sect. VIII.12]{bpv}.

\section{Polarized K3 Surfaces} \label{sect:polarization}

One way to solve this problem is to restrict our attention to a subclass of K3 surfaces that have better properties: the \emph{pseudo-polarized} K3 surfaces.

\begin{definition} \label{polarizeddefn} A (\emph{pseudo}-)\emph{polarized K3 surface of degree $2k$} (for $k>0$) is a pair $(S,h)$ consisting of a K3 surface $S$ and a primitive (pseudo-)ample class $h \in \NS(S)$ with $\langle h,h\rangle = 2k$.

Two  (pseudo-)polarized K3 surfaces $(S,h)$ and $(S',h')$ of degree $2k$ are \emph{equivalent} if there exists an isomorphism $f\colon S \to S'$ of K3 surfaces such that $f^*(h') = h$.
\end{definition}

\begin{remark} \label{remark:projective} If $\NS(S)$ contains a pseudo-ample class, then $S$ is a Moishezon manifold by \cite[Thm. 2.2.15]{hmibk}. As $S$ is also K\"{a}hler, \cite[Thm. 2.2.26]{hmibk} implies that $S$ is projective. Thus every pseudo-polarized K3 surface is projective.
\end{remark}

The geometry of pseudo-polarized K3 surfaces was studied by Mayer \cite{fk3s}. The following easy consequence of Props. 1 and 2 from his paper is particularly useful for studying them explicitly.

\begin{proposition} \label{prop:mayer} Let $(S,h)$ be a pseudo-polarized K3 surface of degree $2k$ and let $D$ be an effective divisor on $S$ with $[D] = h$ in $\NS(S)$. Then the map $f\colon S \to \Proj(H^0(S,\calO_S(D)))$ defined by the linear system $|D|$ is
\begin{itemize}
\item \textup{(}generic case\textup{)} a birational morphism onto a normal surface of degree $2k$ in $\Proj^{k+1}$ if the general member of $|D|$ is a smooth non-hyperelliptic curve; or
\item \textup{(}hyperelliptic case\textup{)} a morphism of degree $2$ onto a normal surface of degree $k$ in $\Proj^{k+1}$ if the general member of $|D|$ is a smooth hyperelliptic curve; or
\item \textup{(}unigonal case\textup{)} a regular map $S \to \Proj^{k+1}$ whose image is a rational curve of degree $k+1$ if the general member of $|D|$ is reducible.
\end{itemize}
\end{proposition}

Using this, we can introduce two of the most widely studied classes of K3 surfaces.

\begin{example}[Sextic double planes] \label{ex:deg2} Suppose first that $(S,h)$ is a pseudo-polarized K3 surface of degree $2$ and let $D$ be an effective divisor on $S$ with $[D] = h$. Then the general member of $|D|$ is either a smooth hyperelliptic curve or is reducible.

In the hyperelliptic case, which for degree $2$ is generic, the linear system $|D|$ defines a generically $2:1$ map from $S$ onto $\Proj^2$. We thus see that $S$ is birational to a double cover of $\Proj^2$ ramified over a sextic curve. Such surfaces may be realized as sextic hypersurfaces in the weighted projective space $\mathbb{W}\Proj(1,1,1,3)$.

In the unigonal case, which for degree $2$ can only occur when $D$ is pseudo-ample but not ample, the linear system $|D|$ defines a regular map from $S$ onto a smooth conic in $\Proj^2$. The general fibre of this map is a smooth elliptic curve.
\end{example}

\begin{example}[Quartic hypersurfaces] \label{ex:deg4} For our second example, suppose that $(S,h)$ is a pseudo-polarized K3 surface of degree $4$ and let $D$ be an effective divisor on $S$ with $[D] = h$. Then all three cases from Prop. \ref{prop:mayer} can occur.

In the generic case the linear system $|D|$ defines a birational morphism onto a quartic hypersurface in $\Proj^3$.

In the hyperelliptic case, the linear system $|D|$ defines a generically $2:1$ map from $S$ onto a quadric hypersurface in $\Proj^3$. This hypersurface is isomorphic to $\Proj^1 \times \Proj^1$, so $S$ is birational to a double cover of $\Proj^1 \times \Proj^1$ ramified over a curve of bidegree $(4,4)$. Such surfaces may be realized as complete intersections of degree $(2,4)$ in the weighted projective space $\mathbb{W}\mathbb{P}(1,1,1,1,2)$, where the degree two relation does not involve the degree two variable (since if it did then we could eliminate it, putting us back in the generic case of a quartic hypersurface in $\Proj^3$).

Finally, in the unigonal case the linear system $|D|$ defines a regular map from $S$ onto a twisted cubic in  $\Proj^3$. The general fibre of this map is again a smooth elliptic curve.
\end{example}

\subsection{Moduli of Polarized K3 Surfaces} \label{sect:polarizedmoduli}

For polarized K3 surfaces we have an upgraded version of the Weak Torelli Theorem (Thm. \ref{weaktorelli}), which will enable us to build a moduli space for them.

\begin{theorem}[Strong Torelli] \textup{\cite[Cor. VIII.3.12 and Thm. VIII.11.1]{bpv}} \label{strongtorelli}  Let $(S,h)$ and $(S',h')$ be polarized K3 surfaces of the same degree $2k$. Assume that there is a Hodge isometry $\varphi\colon H^2(S',\Z) \to H^2(S,\Z)$ with $\varphi(h')=h$. Then there is a \emph{unique} isomorphism $f\colon S \to S'$ with $\varphi = f^*$ \textup{(}i.e. $S$ and $S'$ are equivalent\textup{)}. 
\end{theorem}

Following our previous discussion, we next construct a period space for pseudo-polarized K3 surfaces. Fix once and for all a primitive class $h \in \Lambda_{\mathrm{K}3}$ with $\langle h,h\rangle = 2k >0$. Then a \emph{marked} \mbox{(\emph{pseudo}-)}\emph{polarized K3 surface of degree $2k$} is a marked K3 surface $(S,\phi)$ such that  $\phi^{-1}(h)$ is a (pseudo-)ample class in $\NS(S)$.

If $(S,\phi)$ is a marked pseudo-polarized K3 surface of degree $2k$ and if $\sigma \in H^{2,0}(S)$ is any nonzero element, then we have $\langle\sigma,\sigma\rangle = 0$, $\langle\sigma,\overline{\sigma}\rangle > 0$ and $\langle\sigma,\phi^{-1}(h)\rangle = 0$. So the period point of $(S,\phi)$ lies in
\[\Omega_{2k} := \{[\sigma]\in \Proj(\Lambda_{\mathrm{K}3} \otimes \C) \mid \langle\sigma,\sigma\rangle = 0,\ \langle\sigma,\overline{\sigma}\rangle > 0,\ \langle\sigma,h\rangle = 0\} \subset \Omega_{\mathrm{K}3}.\]
$\Omega_{2k}$ is called the \emph{period space of pseudo-polarized K3 surfaces of degree $2k$}. It is a $19$-dimensional complex manifold with two connected components, each of which is a bounded symmetric domain of type IV \cite[Rmk. VIII.22.2]{bpv}.

By the Surjectivity of the Period Map (Thm. \ref{surjectivity}), every point of $\Omega_{2k}$ corresponds to a marked K3 surface $(S,\phi)$. Furthermore, for any generator $\sigma \in H^{2,0}(S)$ we have $\langle \sigma, \phi^{-1}(h) \rangle = 0$, so $\phi^{-1}(h) \in \NS(S)$. Thus if $\phi^{-1}(h)$ is an ample class, then $(S,\phi)$ will be a marked polarized K3 surface of degree $2k$.

Using this and the Torelli Theorems (Thms. \ref{localtorelli} and \ref{strongtorelli}) we can construct a coarse moduli space for polarized K3 surfaces of degree $2k$. First, however, we have to remove the points in $\Omega_{2k}$ corresponding to the K3 surfaces that are pseudo-polarized but not polarized. 

If a marked pseudo-polarized K3 surface $(S,\phi)$ is not polarized, then the Hodge Index Theorem \cite[Cor. IV.2.16]{bpv} and the genus formula imply that it must contain a rational curve $C$, such that the class $\delta$ of $C$ in $\NS(S)$ satisfies $\langle \delta, \delta \rangle = -2$ and $\langle \delta, \phi^{-1}(h) \rangle = 0$. The converse is also true: if $\NS(S)$ contains such a $\delta$, then by \cite[Prop. VIII.3.7]{bpv} there exists a rational curve $C$ on $S$ with $\langle[C],\phi^{-1}(h)\rangle = 0$, so $\phi^{-1}(h)$ is not ample and $(S,\phi)$ is not polarized. Using this, we see that a marked pseudo-polarized K3 surface $(S,\phi)$ is not polarized if and only if its period point $[\sigma]$ is orthogonal to a point in the set
\[ \Delta_{2k} := \{\delta \in  \Lambda_{\mathrm{K}3} \mid \langle \delta, \delta \rangle = -2, \ \langle \delta, h \rangle = 0 \}.\]

For each $\delta \in \Delta_{2k}$, define a hyperplane
\begin{equation} \label{eq:H} H_{\delta} := \{[\sigma]\in  \Proj(\Lambda_{\mathrm{K}3} \otimes \C) \mid \langle \sigma, \delta \rangle = 0\}.\end{equation}
Then define
\[ \Omega_{2k}^0 := \Omega_{2k} - \bigcup_{\delta \in \Delta_{2k}} (H_{\delta} \cap \Omega_{2k}).\]
$ \Omega_{2k}^0$ is called the \emph{period space of polarized K3 surfaces}. The Torelli Theorems (Thms. \ref{localtorelli} and \ref{strongtorelli}) and the Surjectivity of the Period Map (Thm. \ref{surjectivity}) show that its points are in bijection with marked polarized K3 surfaces of degree $2k$.

It just remains to quotient by an appropriate group to identify period points corresponding to different markings on the same K3 surface. Let $\Gamma(h)$ denote the group of isometries of $\Lambda_{\mathrm{K}3}$ that fix the class $h$. Then $\Gamma(h)$ acts properly discontinuously on $\Omega_{2k}^0$, so the quotient $\Gamma(h) \setminus \Omega_{2k}^0$ will not have the same problems that we experienced before. Thus we have:

\begin{theorem} \textup{\cite[Thm. VIII.22.4]{bpv}} The quotient
\[\calF_{2k}^0 := \Gamma(h) \setminus \Omega_{2k}^0\]
is a coarse moduli space for polarized K3 surfaces of degree $2k$.
\end{theorem}

\begin{remark} At first glance this definition appears to depend upon the choice of the class $h \in \Lambda_{\mathrm{K}3}$. However, it can be shown that all choices of $h$ yield the same moduli space and that, in fact, it is possible to construct $\calF_{2k}$ without making reference to a specific choice of $h$.  However, this construction requires somewhat more theoretical background than the one presented above, so we will not detail it here: the interested reader may refer to \cite[Sect. 1.1]{cmsak3s} for details.
\end{remark}

$\calF_{2k}^0$ is a $19$-dimensional complex space with only finite quotient singularities. By exhibiting a projective compactification, Baily and Borel \cite{caqbsd} showed that it is even quasi-projective (their compactification will be discussed further in Sect. \ref{sect:compactification}). It may be obtained from $\calF_{2k} := \Gamma(h) \setminus \Omega_{2k}$ by removing finitely many divisors. 

$\calF_{2k}$ is also a  $19$-dimensional quasi-projective variety with only finite quotient singularities and, as $\Gamma(h)$ contains an isometry that interchanges the two connected components of $\Omega_{2k}$, it is even connected \cite[Thm. VIII.22.3]{bpv}. One can think of the points in the complement $\calF_{2k} - \calF_{2k}^0$ as corresponding to K3 surfaces that are pseudo-polarized but not polarized, but the definitions required to make this rigourous are somewhat subtle; we refer the interested reader to \cite{srmk3s}.

$\calF_{2k}$ may therefore be thought of as a coarse moduli space for pseudo-polarized K3 surfaces of degree $2k$. It can be realized as a quotient of a bounded symmetric domain of type IV (given by one of the connected components of $\Omega_{2k}$), by an arithmetically defined discrete group of automorphisms, a fact that makes it very amenable to explicit study: details may be found in \cite[Sect. 1]{cmsak3s}.

\subsection{Degenerations}

In the remainder of this section, we will discuss what happens when we proceed to the boundary of this moduli space. In order to do this we study \emph{degenerations}. Our main reference for this study will be \cite{bgd1}.

\begin{definition} A \emph{degeneration of K3 surfaces} is a proper, flat, surjective morphism $\pi\colon \calS \to \Delta$ from a smooth threefold $\calS$ to the unit disc $\Delta \subset \C$, whose general fibre $S_t = \pi^{-1}(t)$ for $t\neq 0$ is a smooth K3 surface. Note that we do not assume that $\calS$ is algebraic, but we will make the assumption that the components of the central fibre $S_0 = \pi^{-1}(0)$ are K\"{a}hler.\end{definition}

Let $\pi\colon \calS \to \Delta$ be any degeneration of K3 surfaces. We begin our analysis by converting this degeneration into a form that has certain desirable properties. The first step is to arrange for \emph{semistability}, using the Semistable Reduction Theorem of Knudsen, Mumford and Waterman:

\begin{theorem}[Semistable Reduction] \textup{\cite{tei}} \label{thm:ssr} Let $\pi \colon \calS \to \Delta$ be a degeneration of surfaces. Then there exists an $m$ such that, if $\pi'\colon \calS' \to \Delta$ is the base change by the map $\varrho\colon \Delta \to \Delta$ given by $\varrho(t) = t^m$, there is a birational morphism $\hat{\calS} \to \calS'$ so that $\psi\colon \hat{\calS} \to \Delta$ is semistable, i.e. $\hat{\calS}$ is nonsingular and $\hat{S}_0 :=\psi^{-1}(0)$  is a reduced divisor with normal crossings.
\[\xymatrix{
\hat{\calS} \ar[r] \ar[d]_{\psi}& \calS' \ar[r] \ar[d]_{\pi'} &  \calS \ar[d]_{\pi} \\
\Delta \ar@{=}[r] & \Delta \ar[r]^{\varrho} & \Delta   \\
}\]
\end{theorem}

\begin{remark} Note that this theorem holds for degenerations of surfaces in general, not just for degenerations of K3 surfaces, although we will only use the K3 version here.\end{remark}

To illustrate the computation of semistable reduction in an example, we will simplify matters by considering a degeneration of elliptic curves. The basic theory is largely unchanged from the K3 surface case, but the equations are substantially simpler.

\begin{example}[The cuspidal elliptic curve] Consider the family $\pi\colon \calE \to \Delta := \{t \in \C \mid |t|<\varepsilon\}$ of elliptic curves in $\mathbb{A}^2 \times \Delta$ given by the equation
\[y^2 = x^3 + tf_3(x),\]
where $(x,y)$ are coordinates on $\mathbb{A}^2$ and $f_3(x)$ is a smooth cubic polynomial in $x$ with $f_3(0) \neq 0$.

$\calE$ is smooth, but the central fibre $E_0 = \pi^{-1}(0)$ of $\calE$ is a cuspidal elliptic curve, which does not have normal crossings. $\calE \to \Delta$ is thus not a semistable degeneration of elliptic curves.

To make it semistable, we use Thm. \ref{thm:ssr}. However, first we need to determine the order $m$ of the cover $\varrho$ that we need to take. To do this, first blow up $\calE$ until the central fibre has only normal crossings. Let $m_1,\ldots,m_n$ denote the multiplicities of the irreducible components of the new central fibre. Then $m = \mathrm{lcm}(m_1,\ldots,m_n)$.

In our case, to obtain a fibre with normal crossings we need to blow up the point $(x,y;t) = (0,0;0)$ three times. The strict transform of $E_0$ under this blow up has multiplicity $1$, and the three exceptional curves have multiplicities $2$, $3$ and $6$. We thus have $m = 6$.

Let $\calE'$ denote the pull-back of $\calE$ by the map $\varrho\colon \Delta \to \Delta$ given by $\varrho(t) = t^6$. Then $\calE'$ is given in $\mathbb{A}^2 \times \Delta$ by
\[y^2 = x^3 + t^6f_3(x).\]
$\calE'$ is singular at $(x,y;t) = (0,0;0)$. The singularity is locally analytically isomorphic to $\{y^2 = x^3 + t^6\} \subset \C^3$. This is an example of a \emph{minimally elliptic singularity}. Such singularities have been studied by Laufer \cite{omes}.

The resolution of this singularity is given in \cite[Table 5.1]{omes}. To resolve it, we blow up the point $(0,0;0) \in \calE'$ once. The resulting exceptional curve is an elliptic curve with self-intersection $(-1)$. The resolved family $\hat{\calE} \to \Delta$ is semistable, with central fibre consisting of a rational $(-1)$-curve meeting an elliptic $(-1)$-curve at a single node, both with multiplicity $1$.

In fact, in this case we can go one step further, by contracting the rational $(-1)$-curve in the central fibre of $\hat{\calE}$. This does not introduce any new singularities into $\hat{\calE}$, so the resulting family is semistable and all of its fibres are smooth elliptic curves.
\end{example}

Once our denegeration of K3 surfaces is semistable, we may additionally arrange for the canonical bundle of the total space to be trivial, using the following theorem of Kulikov, Persson and Pinkham:

\begin{theorem} \textup{\cite{kul1}\cite{kul2}\cite{dstcb}} If $\psi\colon \hat{\calS} \to \Delta$ is a semistable degeneration of K3 surfaces, and if all components of $\hat{S}_0 = \psi^{-1}(0)$ are K\"{a}hler, then there exists a birational modification $\hat{\calS}'$ of $\hat{\calS}$ such that $\psi'\colon \hat{\calS}' \to \Delta$ is semistable, isomorphic to $\hat{\calS}$ over $\Delta - \{0\}$, and has $\omega_{\hat{\calS}'} \cong \calO_{\hat{\calS}'}$. \end{theorem}

Motivated by this theorem, a \emph{Kulikov model} is defined to be a semistable degeneration of K3 surfaces $\pi\colon \calS \to \Delta$ with $\omega_{\calS} \cong \calO_{\calS}$; the discussion above shows that any degeneration of K3 surfaces may be converted into a Kulikov model by a base change and a birational modification. 

\begin{remark} \label{rem:nonalg} It is important to note that the construction of the Kulikov model is very non-algebraic in nature, so even if $\hat{\calS}$ is algebraic, its Kulikov model $\hat{\calS}'$ may not be. We do, however, know that the Kulikov model $\hat{\calS}'$ is complex analytic and that all components of its central fibre are K\"{a}hler.
\end{remark}

Kulikov models are useful because there exists a rough classification of their central fibres, first proven by Kulikov, Persson, Friedman and Morrison. However, in order to state it we first need to introduce the \emph{dual graph} of the central fibre of a degeneration.

\begin{definition} Let $S_0 = \bigcup V_i$ be the central fibre in a semistable degeneration. Define the dual graph $\Gamma$ of $S_0$ as follows: $\Gamma$ is a simplicial complex whose vertices $P_1, \ldots , P_r$ correspond to the components $V_1,\ldots,V_r$ of $S_0$; the $k$-simplex $\langle P_{i_0},\ldots,P_{i_k} \rangle$ belongs to $\Gamma$ if and only if $V_{i_0} \cap \cdots \cap V_{i_k} \neq \emptyset$. \end{definition}

This enables us to state:

\begin{theorem}[Classification of Kulikov Models] \textup{\cite{bgd1}\cite{kul1}\cite{odas}} \label{thm:kulclass} Let $\pi\colon \calS \to \Delta$ be a semistable degeneration of K3 surfaces with $\omega_{\calS} \cong \calO_{\calS}$, such that all components of $S_0 = \pi^{-1}(0)$ are K\"{a}hler. Then either
\begin{enumerate}[\textup{(}Type I\textup{)}]
\item $S_0$ is a smooth K3 surface;
\item $S_0$ is a chain of elliptic ruled components with rational surfaces at each end, and all double curves are smooth elliptic curves;
\item $S_0$ consists of rational surfaces meeting along rational curves which form cycles in each component. If $\Gamma$ is the dual graph of $S_0$, then $\Gamma$ is a triangulation of the $2$-sphere.
\end{enumerate} \end{theorem}

These cases can also be distinguished by the action of monodromy on the second cohomology $H^2(S_t,\Z)$ of a general fibre. Let $T$ denote the Picard-Lefschetz transformation on $H^2(S_t,\Z)$ obtained by the action of monodromy around $0$ and let $N = \log T$. Then $N$ is nilpotent and has $N = 0$ if $S_0$ is Type I, $N^2 = 0$ and $N\neq 0$ if $S_0$ is Type II, and $N^3 = 0$ and $N^2 \neq 0$ if $S_0$ is Type III.

We conclude this section by giving two examples of degenerations of K3 surfaces, one of Type II and one of Type III.

\begin{example}[Type II degeneration] \label{ex:type2} We begin with the Type II example. Consider the family $\pi\colon \calS \to \Delta := \{t \in \C \mid |t| < \varepsilon\}$ given in $\mathbb{W}\Proj(1,1,1,3) \times \Delta$ by the formula
\[y^2 = (f_3(x_1,x_2,x_3))^2 + tg_6(x_1,x_2,x_3),\]
where $(x_1,x_2,x_3,y)$ are coordinates on $\mathbb{W}\Proj(1,1,1,3)$ of weights $(1,1,1,3)$ respectively and $f_3$, $g_6$ are generic homogeneous polynomials in the $x_i$ of degrees $3$ and $6$ respectively.

The general fibre of $\calS$ is a sextic hypersurface in $\mathbb{W}\Proj(1,1,1,3)$ which, by Example \ref{ex:deg2}, is a generic polarized K3 surface of degree two. The central fibre $S_0 = \pi^{-1}(0)$ is isomorphic to two copies of $\Proj^2$ glued along an elliptic curve $\{f_3(x_1,x_2,x_3) = 0\}\subset \Proj^2$.

Note that $\omega_{\calS} \cong \calO_{\calS}$, but $\calS$ is \emph{not} a Kulikov model as $\calS$ is not smooth, so the family $\calS \to \Delta$ is not semistable. In fact, $\calS$ has eighteen singularities at the points $\{f_3(x_1,x_2,x_3) = g_6(x_1,x_2,x_3) = y = t = 0\}$. Each of these singularities is locally analytically isomorphic to $\{y^2 = x^2 + tz\} \subset \C^4$, which is a threefold node.

To solve this, one's first instinct would be to blow up each of the eighteen singularities individually. This introduces eighteen exceptional divisors $E_1,\ldots,E_{18}$, each of which is isomorphic to $\Proj^1 \times \Proj^1$. Indeed, the resulting family is semistable, but it is still not a Kulikov model: in this case, the canonical bundle is isomorphic to $\calO(E_1 + \cdots + E_{18})$, which is non-trivial.

Instead, we notice that $\calS$ contains two Weil divisors that are not Cartier, given by $\{y \pm f_3(x_1,x_2,x_3) = t = 0\}$. Choosing one of these divisors to blow up, we find that the resolved family $\calS' \to \Delta$ is semistable and the exceptional locus is eighteen copies of $\Proj^1$. As this resolution has not introduced any new divisors (it is an example of a \emph{small resolution}: a resolution with exceptional locus of codimension $\geq 2$), we must have $\omega_{\calS'} \cong \calO_{\calS'}$, so $\calS'$ is a Kulikov model. Its central fibre is a copy of $\Proj^2$ glued to a rational surface (obtained by blowing up $\Proj^2$ at eighteen points) along a smooth elliptic curve. This is an example of a Type II degeneration from Thm. \ref{thm:kulclass}.

At this point we make a crucial note: when we performed the resolution to go from $\calS$ to $\calS'$, we had a choice of which divisor to blow up. This illustrates an important point, which is that \emph{Kulikov models are not unique}. In fact, if we allow ourselves to perform analytic blow ups (and, as noted in Rmk. \ref{rem:nonalg}, in certain cases we have to, as an algebraic resolution with trivial canonical bundle will not always exist) then the situation gets much worse, as we have to make a choice of which divisor to blow up \emph{locally} in a neighbourhood of each node. In our example above this gives $2^{18}$ possible analytic Kulikov models!
\end{example}

\begin{example}[Type III degeneration] Next we look at a Type III example. Consider the family $\pi\colon \calS \to \Delta := \{t \in \C \mid |t| < \varepsilon\}$ given in $\Proj^3 \times \Delta$ by the formula
\[wxyz + tf_4(w,x,y,z) = 0,\]
where $(w,x,y,z)$ are coordinates on $\Proj^3$ and $f_4$ is a generic homogeneous polynomial of degree $4$ in $(w,x,y,z)$.

The general fibre of $\calS$ is a quartic hypersurface in $\Proj^3$ which, by Ex. \ref{ex:deg4}, is a generic polarized K3 surface of degree four. The central fibre $S_0 = \pi^{-1}(0)$ is isomorphic to four copies of $\Proj^2$ given by the coordinate hyperplanes in $\Proj^3$.

As in the previous example, we have $\omega_{\calS} \cong \calO_{\calS}$, but $\calS$ is singular and so not a Kulikov model. There are $24$ singularities, occurring at the intersections of $\{f_4(w,x,y,z) = 0\}$ with the six lines $\{w = x = 0\}$, $\{w = y = 0\}$, etc. As before, each of these singularities is locally analytically isomorphic to a threefold node $\{wx + tz = 0\} \subset \C^4$.

We may resolve these singularities in the same way as the previous example to get a Kulikov model. The central fibre consists of four rational surfaces meeting along six rational curves, with dual graph homeomorphic to a tetrahedron. This is an example of a Type III degeneration from Thm. \ref{thm:kulclass}.
\end{example}

\subsection{Compactifications} \label{sect:compactification}

Given that we have such a good description of the moduli space for pseudo-polarized K3 surfaces, it is natural to ask whether there is a nice way to compactify this moduli space, i.e. find a compact variety $\overline{\calF}_{2k}$ that contains $\calF_{2k}$ as an open subset. Preferably, one would like to do this in such a way that the boundary $\overline{\calF}_{2k} - \calF_{2k}$ encodes some geometric data about the corresponding degenerate K3 surfaces (in fact, ideally, we would like the boudary to provide moduli for degenerate K3's).

Probably the best known compactification of $\calF_{2k}$ is the \emph{Baily-Borel compactification}, first constructed in \cite{caqbsd}. This is a method to compactify any arithmetic quotient of a bounded symmetric domain, of which $\calF_{2k}$ is an example. In the case of $\calF_{2k}$, this compactification was studied in detail by Scattone \cite{cmsak3s}. Its boundary is a union of $0$- and $1$-dimensional strata, which have some geometric meaning: the $0$-dimensional strata correspond to degenerate K3's of Type III, and the $1$-dimensional strata to degenerate K3's of Type II. Furthermore, the $1$-dimensional strata are all rational curves, each of which is parametrised by the $j$-invariant of the elliptic double curves appearing in the corresponding Type II degenerate K3. 

\begin{example}[K3 surfaces of degree $2$] In the case of $\calF_2$, Friedman \cite[Sect. 5]{npgttk3s} showed that the boundary of the Baily-Borel compactification has four components of dimension $1$, which meet in a unique $0$-dimensional component. Furthermore, he gave a coarse geometric description of the degenerate fibres corresponding to each Type II boundary component; one of the four possibilities is given by the central fibre in Ex. \ref{ex:type2}.
\end{example}

\begin{example}[K3 surfaces of degree $4$] In the case of $\calF_4$, Scattone \cite[Sect. 6]{cmsak3s} has shown that the boundary of the Baily-Borel compactification has nine $1$-dimensional components, which meet in a unique $0$-dimensional component. This serves to illustrate that the number of boundary components in the Baily-Borel compactification of $\calF_{2k}$ grows very quickly with $k$; in fact, Scattone \cite{cmsak3s} has shown that it grows like $k^8$.
\end{example}

Unfortunately this is about all one can say about the Baily-Borel compactification: the boundary is simply too small to encode more detailed geometric data about degenerate K3's, let alone provide moduli for them. 

Several other compactifications also exist. \emph{Toroidal compactifications}, originally constructed by Mumford \cite{sclcv}\cite{tcss}, can be described as toroidal blow-ups of the boundary in the Baily-Borel compactification; a brief overview of this construction may be found in \cite[Sect. 2.2]{cmsak3s}. They have the advantage of being fairly easy to construct and their boundary components admit a clean explicit description. Furthermore, they can encode more detailed geometric data about degenerate K3's than is found in the Baily-Borel compactification (for instance, see \cite{pmbm}).  However, there is a large collection of such compactifications, corresponding to possible choices of blow-up, with no clear canonical choice amongst them.

One may also compactify using tools from the log minimal model program developed by Koll\'{a}r, Shepherd-Barron and Alexeev \cite{lcscmsp}\cite{msmgnws}\cite{hdasc}\cite{tadoss}. Such so-called \emph{KSBA compactifications} avoid many of the problems encountered by toroidal compactifications, but come with some of their own. They can be defined in a way that avoids choices, making them in some sense canonical, and their construction ensures that their boundary provides moduli for degenerate K3 surfaces. However, it is difficult to describe the boundaries of KSBA compactifications explicitly, making them hard to study.

Probably the best compactifications currently known are those that use techniques from Geometric Invariant Theory. Such \emph{GIT compactifications} were originally studied by Shah \cite{cmsk3sd2}\cite{dk3sd4} and, much more recently, Laza \cite{ksbacmsdtk3p}. These solve both problems: they admit clean explicit descriptions and their boundaries provide moduli for degenerate K3 surfaces. However, they are difficult to construct in general, so thus far have only been studied for small values of $k$: \cite{ksbacmsdtk3p} and \cite{cmsk3sd2} treat only the case $k = 1$, and \cite{dk3sd4} treats the case $k = 2$.

\section{Lattice Polarized K3 Surfaces}
\label{section:polarizationmoduli}

This general theory is all well and good, but so far we have relatively few explicit examples of K3 surfaces to work with (the polarized K3 surfaces of Exs. \ref{ex:deg2} and \ref{ex:deg4} are effectively the only ones we have constructed). Many more interesting families of K3 surfaces may be studied using the machinery of lattice polarizations. 

The concept of lattice polarization extends the idea of polarization discussed in Sect. \ref{sect:polarization}.  It was first introduced by Dolgachev \cite{mslpk3s}, whose results we will summarize here. We begin with a central definition.

\begin{definition}
\label{definition:lattice-polarization}
Let $S$ be an projective K3 surface and let $L$ be a non-degenerate lattice. Assume that there is a lattice embedding $\iota \colon L \hookrightarrow \NS(S)$. Then we say that the pair $(S,\iota)$ is a (\emph{pseudo-})\emph{ample $L$-polarized K3 surface} if 
\begin{enumerate}
\item The embedding $\iota$ is primitive, and
\item The image of $\iota$ contains an (pseudo-)ample class.
\end{enumerate}

Two (pseudo-)ample $L$-polarized K3 surfaces $(S,\iota)$ and $(S',\iota')$ are \emph{equivalent} if there is an isomorphism $f\colon S \rightarrow S'$ of K3 surfaces such that $f^* \circ \iota' = \iota$ and $f^*$ takes some (pseudo-)ample class on $S'$ to a (pseudo-)ample class on $S$.
\end{definition}

We will often informally use the phrase ``$L$-polarization'' to indicate a pseudo-ample $L$-polarization. The argument from Rmk. \ref{remark:projective} shows that every pseudo-ample $L$-polarized K3 surface is projective.

This definition is a natural generalization of the concept of polarization from Sect. \ref{sect:polarization}. Indeed, if $L$ is a non-degenerate lattice of rank $1$ and if $(S,\iota)$ is a pseudo-ample $L$-polarized K3 surface, then the image $\iota(L)$ contains a pseudo-ample class $h$ with $\langle h, h\rangle > 0$. Furthermore, since $\Lambda_{\mathrm{K}3}$ is even, we must have $\langle h, h\rangle = 2k > 0$, so $(S,h)$ is a pseudo-polarized K3 surface of degree $2k$. More examples of lattice polarized K3 surfaces will be given in Sects. \ref{section:Examples} and \ref{section:fibrations}. 

In general, since $\iota(L)$ contains a pseudo-ample class, $L$ must contain some $h$ with $\langle h,h \rangle > 0$. This, along with the condition that $\iota(L)$ must be contained in $\NS(S)$ and the Signature Theorem \cite[Thm. IV.2.14]{bpv}, shows that $L$ must be a lattice of signature $(1,n-1)$ where $n$ is the rank of $L$.

\subsection{Moduli of Lattice Polarized K3 Surfaces}

We next discuss the moduli space of K3 surfaces with lattice polarization. The construction works in largely the same way as the construction of the moduli space for polarized K3 surfaces discussed in Sect. \ref{sect:polarizedmoduli}.

Fix once and for all a nondegenerate lattice $L$ of rank $n$ and a primitive embedding of $L$ into $\Lambda_{\mathrm{K}3}$. We consider $L$ as a sublattice of $\Lambda_{\mathrm{K}3}$. Then a marked K3 surface $(S,\phi)$ is a \emph{marked \textup{(}pseudo-\textup{)}ample $L$-polarized K3 surface} if the restriction $\phi^{-1}|_L \colon L \rightarrow H^2(S,\mathbb{Z})$ is a (pseudo-)ample $L$-polarization on $S$. 

Note that for an generator $\sigma$ of $H^{2,0}(S)$, we have $\langle \phi_\mathbb{C}(\sigma), u \rangle = 0$ for any $u \in L$, since $\phi^{-1}(u)$ is contained in $\NS(S)$. Thus the period points of marked $L$-polarized K3 surfaces are contained in
\[\Omega_{L^\perp}  := \{[\sigma] \in \mathbb{P}(\Lambda_{\mathrm{K}3} \otimes \C) \mid \langle \sigma,\sigma\rangle = 0,\ \langle \sigma, \overline{\sigma} \rangle > 0,\ \langle \sigma, L \rangle = 0 \} \subset \Omega_{\mathrm{K}3}.\]
$\Omega_{L^\perp}$ called the \emph{period space of pseudo-ample $L$-polarized K3 surfaces}. It is a complex manifold of dimension $(20 - n)$ with two connected components, each of which is a bounded symmetric domain of type IV \cite[Sect. 3]{mslpk3s}.

\begin{remark} \label{remark:typeIV} By \cite[Sect. X.6.3]{dglgss}, a bounded symmetric domain of type IV (also called ``type BD I ($q = 2$)'' in \cite{dglgss}) and complex dimension $n$ is isomorphic to $\mathrm{SO}_0(2,n)/\mathrm{SO}(2) \times \mathrm{SO}(n)$. Furthermore, by the isomorphisms in  \cite[Sect. X.6.4]{dglgss},  in dimensions $1$, $2$ and $3$ these domains coincide with the classical modular domains $\mathbb{H}$, $\mathbb{H} \times \mathbb{H}$ and $\mathbb{H}_2$ respectively, where $\mathbb{H}$ denotes the upper half space in $\mathbb{C}$ and $\mathbb{H}_2$ denotes the Siegel upper half space of genus $2$.
\end{remark}

By the Surjectivity of the Period Map (Thm. \ref{surjectivity}), every point in $\Omega_{L^\perp}$ is the period point of some marked K3 surface $(S,\phi)$. Furthermore, for any such marked K3 surface we still have $\langle \sigma, L \rangle = 0$, so $\phi^{-1}(L)$ is contained in $\NS(S)$. In fact, outside of certain codimension $1$ loci, $\phi^{-1}(L)$ actually contains an ample class on $S$. Define, as before,
\[\Delta_{L^\perp}:= \{ \delta \in \Lambda_{\mathrm{K}3} \mid \langle \delta,\delta\rangle = -2,\ \langle \delta,L \rangle = 0\}.\]
To each $\delta$ in $\Delta_{L^\perp}$, we may assign a hyperplane $H_\delta$ as in Eq. \eqref{eq:H}. Then define
\[\Omega^0_{L^\perp} := \Omega_{L^\perp} - \bigcup_{\delta \in \Delta_{L^\perp}} (H_\delta \cap \Omega_{L^\perp}).\]

$\Omega^0_{L^\perp}$ called the \emph{period space of ample $L$-polarized K3 surfaces}. Dolgachev \cite{mslpk3s} proved that:

\begin{theorem}\textup{\cite[Thm. 3.1 and Cor. 3.2]{mslpk3s}} The points of the space $\Omega^0_{L^\perp}$ \textup{(}resp. $\Omega_{L^\perp}$\textup{)} are in bijection with ample \textup{(}resp. pseudo-ample\textup{)} marked $L$-polarized K3 surfaces.
\end{theorem}

Using  the Torelli Theorems (Thms. \ref{localtorelli} and \ref{strongtorelli}) and the Surjectivity of the Period Map (Thm. \ref{surjectivity}), this theorem is not very difficult to prove; the proof essentially amounts to notation and keeping track of K\"ahler data. We refer the interested reader to \cite[Sect. 2]{mslpk3s} for details.

To construct a coarse moduli space for (pseudo-)ample $L$-polarized K3 surfaces, we once again have to perform a quotient to get rid of the choice of marking. Let $\Gamma_{L^\perp}$ be the subgroup of elements of $\mathrm{O}(\Lambda_{\mathrm{K}3})$ satisfying
\[\Gamma_{L^\perp} := \{ \gamma \in \mathrm{O}(\Lambda_{\mathrm{K}3}) \mid \gamma |_L = \Id\},\]
i.e. the subgroup of elements of $\mathrm{O}(\Lambda_{\mathrm{K}3})$ which fix $L \subseteq \Lambda_{\mathrm{K}3}$ pointwise. Then we find:

\begin{theorem}\textup{\cite[Rmk. 3.4]{mslpk3s}} The quotients
\[\mathcal{F}_{L^\perp}^0 := \Gamma_{L^\perp} \setminus \Omega_{L^\perp}^0 \]
and
\[\mathcal{F}_{L^\perp}:=\Gamma_{L^\perp} \setminus \Omega_{L^\perp}\]
are coarse moduli spaces for ample and pseudo-ample $L$-polarized K3 surfaces respectively.
\end{theorem}

As before, $\mathcal{F}_{L^\perp}$ may be seen as an arithmetic quotient of a bounded symmetric domain of type IV (given by one of the connected components of $\Omega_{L^{\perp}}$), so $\mathcal{F}_{L^\perp}$ is connected and one may use the work of Baily-Borel \cite{caqbsd} to show that both $\mathcal{F}^0_{L^\perp}$ and $\mathcal{F}_{L^\perp}$ are quasi-projective varieties. Note that $\mathcal{F}_{L^\perp}$ has dimension $(20 - n)$.

\begin{example}[Lattices of rank 19]
If $L$ has rank $19$ then we see that $\Omega_{L^\perp}$ is a $1$-dimensional space so, by Rmk. \ref{remark:typeIV}, each connected component of $\Omega_{L^\perp}$ is analytically isomorphic to  $\mathbb{H}$, the upper half plane in $\mathbb{C}$. We thus see that $\calF_{L^\perp}$ is isomorphic to the quotient of $\mathbb{H}$ by the action of the group $\Gamma_{L^\perp}$. It can be shown that this group acts on $\mathbb{H}$ as a discrete subgroup of $\mathrm{SL}_2(\mathbb{R})$.

For instance, if $L \cong H \oplus (-E_8) \oplus (-E_8) \oplus \left<-2n\right>$, then \cite[Thm. 7.1]{mslpk3s} shows that the group $\Gamma_{L^\perp}$ acts on $\mathbb{H}$ in the same way as $\Gamma_0(n)^+$. Thus 
\[\mathcal{F}_{L^\perp}  = \Gamma_0(n)^+ \setminus \mathbb{H},\]
which is a classical modular curve.

A more complicated example is provided by the lattice $L \cong H \oplus (-E_8) \oplus (-D_7) \oplus (-A_2)$. In this case, Elkies \cite[Sect. 3]{scck3snsr19} has proved that $\mathcal{F}_{L^\perp}$ is isomorphic to a quotient of the \emph{Shimura curve} $X(6)$. We refer the interested reader to \cite{scck3snsr19} for full details of this and several other Shimura curve examples.
\end{example}

\begin{example}[Lattices of rank 18]
If $L$ is a lattice of rank 18, then we find a similar structure. In this case, by Rmk. \ref{remark:typeIV}, each connected component of $\Omega_{L^\perp}$ is analytically isomorphic to the product $\mathbb{H} \times \mathbb{H}$, and $\calF_{L^\perp}$ is isomorphic to the quotient of $\mathbb{H} \times \mathbb{H}$ by the action of $\Gamma_{L^\perp}$. This group acts on $\mathbb{H} \times \mathbb{H}$ as a subgroup of $\left((\mathrm{SL}_2(\mathbb{R} ) \times \mathrm{SL}_2(\mathbb{R})\right) \rtimes \Z / 2\Z$, where the $\Z / 2\Z$ acts to exchange the factors of $\mathbb{H} \times \mathbb{H}$.

In particular, if 
\[L \cong \left( \begin{matrix} 2 & D \\ D & (D^2 - D)/2 \end{matrix} \right) \oplus (-E_8) \oplus (-E_8)\]
for $D$ square-free, then $\mathcal{F}_{L^\perp}$ is birational to the classically known \emph{Humbert surface of discriminant $D$} by \cite[Thm. 9]{k3sehms}. 

If $L \cong H \oplus (-E_8) \oplus (-E_8)$, then $\Gamma_{L^\perp}$ is isomorphic to $\left(\mathrm{SL}_2(\mathbb{Z}) \times \mathrm{SL}_2(\mathbb{Z})\right) \rtimes \Z / 2\Z$. In this case $\mathcal{F}_{L^\perp}$ is isomorphic to the symmetric product of two copies of the classical modular curve; see \cite{milpk3s} or \cite{nfk3smmp} for more details.
\end{example}

One may notice that the moduli spaces of $L$-polarized K3 surfaces with $\rank L \geq 17$ are often related to the moduli spaces of abelian surfaces. This is no accident. According to \cite{k3slpn}, K3 surfaces with lattice polarization by $L = (-E_8) \oplus (-E_8) \oplus \langle 2n \rangle$ can be constructed geometrically from abelian surfaces with $(1,n)$ polarization. More generally, Galluzzi and Lombardo \cite{cbk3s} show that a large class of rank $17$ polarized K3 surfaces admit algebraic correspondences with abelian surfaces. This is part of a more general relation between periods of K3 surfaces and abelian varieties called the Kuga-Satake construction, which was first described in \cite{avapk3s}. A modern introduction to this construction may be found in \cite{ksvhc}.

\subsection{Examples of Lattice Polarized K3 Surfaces} \label{section:Examples}

In this subsection we will provide some simple examples of lattice polarized K3 surfaces. The aim is to show that the lattice polarization construction given above arises naturally and geometrically.

\subsubsection{Examples Arising from Embeddings} \label{section:embeddings}

If $S$ is a K3 surface embedded as a subvariety of a smooth projective variety $X$, then there is a natural restriction map $r\colon \NS(X)\otimes \mathbb{Q} \rightarrow \NS(S)\otimes \mathbb{Q}$. Let $L$ be the lattice 
\[L:= \Image r \cap \NS(S) \subseteq \NS(S) \otimes \mathbb{Q}.\] 
We claim that $S$ is $L$-polarized.

The image of this restriction morphism is non-trivial, first of all, since any ample class on $X$ restricts to an ample class on $S$; this also shows that condition 2 of Def. \ref{definition:lattice-polarization} is satisfied. By construction, the lattice $L$ is a primitive sublattice of $\NS(S)$, so condition 1 of Def. \ref{definition:lattice-polarization} is satisfied. Therefore $S$ is $L$-polarized, as claimed.

As an interesting example of this, let $X$ be a smooth Fano threefold (i.e., a threefold whose anticanonical bundle is ample) and let $S$ be a smooth member of the anticanonical linear system $|-K_X|$; by adjunction $S$ is a K3 surface. The Lefschetz Hyperplane Theorem shows that the map $r\colon \NS(X) \rightarrow \NS(S)$ is a a primitive embedding. Furthermore, by \cite{ftk3s}, the map $r$ can be seen as a lattice embedding if we equip $\NS(X)$ with the bilinear form
\begin{equation}
\label{equation:intersection}
\langle u,v \rangle = \langle [-K_X],u,v \rangle_X,
\end{equation}
where $\left<\cdot,\cdot,\cdot\right>_X$ is the trilinear intersection form on $\NS(X)$.

\begin{example}[Anticanonical hypersurfaces in $\mathbb{P}^1 \times \mathbb{P}^2$]
\label{example:P1P2}
For instance, if we take $X = \mathbb{P}^1 \times \mathbb{P}^2$, then $\NS(X) \cong \mathbb{Z}^2$ is spanned by the divisor classes
\[D_1 := [p  \times \mathbb{P}^2], \quad D_2 := [\mathbb{P}^1 \times H]\]
where $p$ is a point in $\mathbb{P}^1$ and $H$ is a line in $\mathbb{P}^2$. One may check that $-K_X \sim i_1^* (-K_{\mathbb{P}^1}) + i_2^*(-K_{\mathbb{P}^2})$ where $i_1$ is the projection map onto $\mathbb{P}^1$ and $i_2$ is the projection map onto $\mathbb{P}^1$. Thus $[-K_X] = 2D_1 + 3D_2$. 

Now it is easy to compute that 
\[\langle D_1, D_1,  D \rangle_X = 0\]
for any divisor $D$, and that
\[\langle D_1, D_2, D_2\rangle_X = 1,\quad \langle D_2,D_2,D_2 \rangle_X = 0.\]
Using Eq. \eqref{equation:intersection} we deduce that any $S$ embedded as an anticanonical hypersurface in $X$ is lattice polarized by a lattice of rank $2$ with Gram matrix
\[\left( \begin{matrix} 0 & 3 \\ 3 & 2  \end{matrix} \right).\]
\end{example}

If $X$ is a smooth toric weak Fano threefold  (i.e. a threefold whose anticanonical bundle is pseudo-ample) then one may effectively compute the induced lattice polarization induced on the anticanonical K3 surface $S$ using toric geometry. A large number of K3 surfaces of this form were found by Reid \cite{c3f}, as toric resolutions of hypersurfaces in weighted projective spaces (the full list is given in \cite[Sect. 13.3]{wwwci}); the induced lattice polarizations on them were computed by Belcastro \cite{plfk3s}. Further examples of this type have been computed by Rohsiepe \cite{lptk3s}\cite{cyhtvfd}.

\subsubsection{Examples Arising from Singularities}

Another useful way of producing lattice polarizations on K3 surfaces is by introducing controlled singularities. In particular, if $\overline{S}$ is a compact algebraic surface with singularities of ADE type (see, for instance \cite[Sect. 4.2]{chapters}) and trivial dualizing sheaf $\omega_{\overline{S}} \cong \calO_{\overline{S}}$, then the minimal resolution $S$ of $\overline{S}$ is a K3 surface and the exceptional divisor associated to each singularity is a configuration of rational curves whose dual intersection graph is a Dynkin diagram of ADE type. The associated (negative definite) root system is then contained inside of the N\'eron-Severi lattice of $S$. 

\begin{example}[Nodal quartics]
\label{example:nodalquartic}
Let $\overline{S}$ be a quartic hypersurface in $\mathbb{P}^3$ with a single $A_1$ singularity and minimal resolution $S \to \overline{S}$. The surface $S$ is a pseudo-polarized K3 surface of degree $4$ but is \emph{not} polarized, since the embedding into $\mathbb{P}^3$ induced by the hyperplane section is not a smooth embedding and so the class in $\NS(S)$ coming from the restriction of the hyperplane class in $\mathbb{P}^3$ is only pseudo-ample.

After resolving the singularity, we obtain an exceptional rational curve $C$ in $S$. The class of $C$ has self-intersection $(-2)$ by the Riemann-Roch Theorem for surfaces. We also see that $\langle [H|_S], [C] \rangle = 0$ for a generic hyperplane section $H$ of $\mathbb{P}^3$. Therefore there is an embedding of the rank two lattice $L$ with Gram matrix
\[\left( \begin{matrix} -2 & 0 \\ 0 & 4 \end{matrix} \right)\]
into $\NS(S)$. 

As we have already observed, the class $[H|_S]$ is pseudo-ample, thus condition 2 of Def. \ref{definition:lattice-polarization} holds. Furthermore, if $L$ were not embedded primitively into $\NS(S)$, then there would be some even integral sublattice of $L\otimes \mathbb{Q}$ containing $L$. One may argue using Thm. \ref{thm:overlattices} from the appendix that no such sublattice exists. Thus $L$ is a primitive sublattice of $\NS(S)$, so condition 1 of Def. \ref{definition:lattice-polarization} is also satisfied. Therefore, we see that $S$ is $L$-polarized.

We should note that, in terms of the moduli space of pseudo-polarized K3 surfaces of degree $4$ constructed in Sect. \ref{sect:polarizedmoduli}, such K3 surfaces lie in one of the hyperplanes $H_\delta$.
\end{example}

In general, one may try to produce lattice polarizations by imposing specific configurations of singularites. However, some care is required. Suppose that $\overline{S}$ is a singular surface with trivial dualizing sheaf  $\omega_{\overline{S}} \cong \calO_{\overline{S}}$, a primitive pseudo-ample class $h \in \NS(\overline{S})$ with $\langle h,h \rangle = 2k$, and ADE singularities of types $R_1, \ldots, R_n$. Then the K3 surface $S$ obtained by resolving the singularities of $\overline{S}$ will not necessarily be lattice polarized by $\langle 2k \rangle \oplus (-R_1) \oplus \dots \oplus (-R_n)$, since this lattice may not be primitively embedded in $\NS(S)$. This phenomenon is illustrated in the following example.

\begin{example}[Kummer surfaces] Suppose that $A$ is an abelian surface. The involution $\iota\colon A \to A$ given by $\iota(x) = -x$ (defined using the group law on $A$) has sixteen fixed points. If we quotient $A$ by this involution, we obtain a projective surface $\overline{S}$ with trivial dualizing sheaf and sixteen singularities of type $A_1$. Each of these singularities may be resolved by blowing up once, giving sixteen disjoint exceptional $(-2)$-curves $E_1,\ldots,E_{16}$. The resolution is a projective K3 surface $S$, called the \emph{Kummer surface associated to $A$}.

Now let $h \in \NS(\overline{S})$ be a primitive ample class and suppose that $\langle h,h \rangle = 2k$. Then from the discussion above, there is an embedding of the lattice $L := \langle 2k \rangle \oplus (-A_1)^{\oplus 16}$ into $\NS(S)$. However, this embedding is not primitive, as we shall now demonstrate.

Begin by noting that there exists a double cover of $S$ branched along the divisor $\sum_{i=1}^{16} E_i$: this cover is precisely the (non-minimal) surface obtained by blowing up the sixteen fixed points of $\iota$ in $A$. However, the existence of this cover is equivalent to the existence of a divisor $D \in \Pic(S)$ with $2D \cong \sum_{i=1}^{16} E_i$. The class $[D]$ of $D$ in $\NS(S)$ is thus equal to $\frac{1}{2}\sum_{i=1}^{16} [E_i]$, which lies in $\frac{1}{2}L$ but not in $L$. So we have found a class $[D] \in \NS(S)$ that lies in $L \otimes \mathbb{Q}$ but not in $L$ and, therefore, the embedding of $L$ into $\NS(S)$ cannot be primitive.
\end{example}

In the next section, we will see that the problem of primitive embeddings has a nice solution in the case where $\overline{S}$ is a singular elliptically fibred K3 surface.

\section{Elliptically Fibred K3 Surfaces} \label{section:fibrations}

In this section, we will give a detailed discussion of the construction of K3 surfaces through elliptic fibrations. As we shall see, this is an excellent source of K3 surfaces with lattice polarizations, but in order to get there we will have to develop some of the theory of elliptically fibred surfaces first. Most of the theory presented here was originally developed by Kodaira \cite{ccasI}\cite{casII}\cite{casIII} and Tate \cite{adtsfep}. A self-contained reference for the reader interested in algebraic elliptic surfaces over $\C$ is Miranda's book \cite{btes}, which we will use as our main reference. However, readers who are interested in elliptic surfaces over arbitrary characteristic might find it useful to consult the more algebraic \cite{esI}, whereas those interested in arithmetic results on elliptic surfaces may find \cite{es} or \cite{ataec} more to their tastes. In addition to this, both \cite{bpv} and \cite{sfmcs} contain discussions of elliptic fibrations on complex manifolds.

\subsection{Elliptic Fibrations and $H$-Polarizations}\label{sect:Hpol}

We begin by studying the construction of elliptic surfaces, with an emphasis on K3 surfaces. Start by letting $S$ be an arbitrary algebraic surface. A \emph{genus one fibration} on $S$ is a pair $(S,\pi)$ where $\pi$ is a projective morphism $\pi \colon S \rightarrow B$ to some smooth curve $B$, with generic fibre a smooth curve of genus one.

We say that $(S,\pi)$ is an \emph{elliptic fibration} if $\pi$ admits a section $\mathbb{O} \colon B \rightarrow S$ such that $\pi \circ \mathbb{O} = \Id$. By the Leray spectral sequence we have $h^1(S,\mathcal{O}_S) \geq h^0(B,\mathcal{O}_B)$, with equality if and only if $S$ is not a product $B \times E$ of $B$ with an elliptic curve $E$ \cite[III.4.1]{btes}. Thus if we want $S$ to be a K3 surface, then $B$ must be $\mathbb{P}^1$.

If $(S,\pi)$ is an elliptic fibration on a K3 surface $S$, then we have two obvious curve classes in $S$. Firstly we have the class of the image of $\mathbb{O}$, which we call $[O]$, and secondly we have the class of a fibre $E$. It is easy to see that 
\[\langle [O], [O] \rangle = -2,\quad  \langle[E],[E] \rangle = 0, \quad \langle [E], [O] \rangle = 1.\]
The sublattice of $\NS(S)$ given by $L:= \mathbb{Z}([O] - [E]) \oplus \mathbb{Z}[E]$ has Gram matrix
\[\left( \begin{matrix} 0 & 1 \\ 1 & 0 \end{matrix} \right),\]
so it is isomorphic to the hyperbolic plane lattice $H$. 

Now, the class $(2[E]+[O])$ in $L$ is pseudo-ample, thus condition 2 of Def. \ref{definition:lattice-polarization} holds. Furthermore, if $L$ were not embedded primitively into $\NS(S)$, then there would be some even integral sublattice of $L\otimes \mathbb{Q}$ containing $L$. Since $L$ is unimodular, Thm. \ref{thm:overlattices} from the appendix shows that no such sublattice exists. Thus $L$ is a primitive sublattice of $\NS(S)$, so condition 1 of Def. \ref{definition:lattice-polarization} is also satisfied. 

We therefore see that any elliptic fibration on a K3 surface $S$ defines an $H$-polarization on $S$. Conversely, it is easy to see that any $H$-polarization on $S$ determines an elliptic fibration. In particular, this shows that any elliptically fibred K3 surface is projective. 

\begin{remark} We remind the reader here that our definition of an elliptic fibration comes with a section: the statements above are not necessarily true for more general genus one fibrations. In particular, a K3 surface with a genus one fibration may not admit an $H$-polarization (see Ex. \ref{ex:genusone}) and does not have to be projective. \end{remark}

\subsection{Singular Fibres}

To find a way to construct elliptic surfaces, we must dig a little deeper into the geometry of an elliptic fibration $\pi\colon S \rightarrow B$. We begin with a definition.

\begin{definition}
An elliptic fibration $\pi\colon S \rightarrow B$ is called \emph{relatively minimal} if it is smooth and there is no rational curve $C$ in $S$ with self-intersection $(-1)$ and $\pi(C)$ a point.
\end{definition}

Essentially, relative minimality means that all curves inside of fibres of $\pi$ that can be contracted smoothly have been contracted. Note that if $S$ is a minimal surface to begin with, it will certainly be a relatively minimal elliptic surface. In particular, this means that any elliptic fibration on a K3 surface is a relatively minimal elliptic fibration. 

On the other hand, consider a pair of cubics $C_1$ and $C_2$ in general position in $\mathbb{P}^2$. The intersection $C_1 \cap C_2$ consists of nine points, through which a pencil of cubics (given by taking projective linear combinations of the defining equations of $C_1$ and $C_2$) passes. If we let $S$ be the blow up of $\mathbb{P}^2$ at these nine points, then $S$ admits an elliptic fibration induced by the pencil of elliptic curves passing through the nine points. This fibration is relatively minimal, but the surface $S$ is not minimal (the nine $(-1)$-curves on $S$ are all sections of the fibration).

We now turn our attention to the fibres of a relatively minimal elliptic fibration $\pi\colon S \to B$. Generically the fibres of $\pi$ are smooth elliptic curves. However at certain points, the fibres of $\pi$ may degenerate to singular curves. The number and type of these singular fibres is somewhat controlled by the following theorem:

\begin{theorem} \textup{\cite[Lemma IV.3.3]{btes}} \label{thm:euler} Suppose that $\pi\colon S \to B$ is a relatively minimal elliptic fibration. Let $e(S)$ be the Euler characteristic of the surface $S$ and let $e(p)$ be the Euler characteristic of a fibre $\pi^{-1}(p)$ of $\pi$. Then
\[e(S) = \sum_{p\in B} e(p).\] 
\end{theorem}

Note that this is actually a finite sum since the Euler characteristic of a smooth fibre is $0$. Since a K3 surface has Euler characteristic $24$, an elliptic fibration on a K3 surface must have singular fibres.

A theorem, originally due to Kodaira \cite{casII}, classifies the singular fibres of smooth minimal elliptic fibrations. This theorem will be presented in full generality in Sect. \ref{section:fibresagain}, but for now the following will suffice.

\begin{theorem}\textup{\cite[Sect. I.4]{btes}}\label{thm:roughfibres}
Let $\pi\colon S \rightarrow B$ be a smooth relatively minimal elliptic fibration. Then any fibre $E$ of $\pi$ is either:
\begin{enumerate}
\item Irreducible, in which case $E$ is either a smooth elliptic curve, or a nodal or cuspidal rational curve; or
\item Reducible, in which case $E$ is a configuration of smooth rational curves $C_i$ with $\langle [C_i],[C_i]\rangle = -2$. In this case $E$ is either a pair of rational curves that are tangent at a point, three rational curves meeting at a single point, or a configuration of rational curves meeting transversely with dual intersection graph of extended ADE type. Furthermore, there are positive integers $a_i$ such that
\begin{equation}
\label{equation:fibres}
[E] =  \sum_{i = 0}^n a_i [C_i].
\end{equation}
\end{enumerate}
\end{theorem}

\subsection{Weierstrass Fibrations}

We have seen that if a singular fibre is not a singular elliptic curve, it is composed of smooth rational curves with self-intersection $(-2)$. If we relabel the curves in each singular fibre so that $C_0$ is the unique rational curve intersecting the section $O$, one sees that the configuration 
\[\sum_{i = 1}^na_i [C_i]\]
has self-intersection $(-2)$ and can thus be contracted to an ADE singularity. Once this contraction has been performed, the component $C_0$ becomes a rational curve with a single node or cusp.

If $(S,\pi)$ is a smooth relatively minimal elliptically fibred surface, then let $(\overline{S},\pi)$ be the singular elliptically fibred surface obtained by contracting components of all singular fibres as above. In this way, we may obtain an elliptically fibred surface whose fibres are all irreducible and have arithmetic genus one. Following \cite[Def. II.3.2]{btes}, we call such a surface a \emph{Weierstrass fibration} and we say that $\overline{S}$ is a \emph{Weierstrass model} for $S$. The rationale for these names will become clear in a moment. 

One may show without much difficulty (see \cite[II.3.5]{btes}) that if $\pi\colon \overline{S} \rightarrow B$ is a Weierstrass fibration, then the sheaf $R^1\pi_* \mathcal{O}_{\overline{S}}$ is a line bundle.

\begin{definition}
Let $\pi\colon \overline{S} \rightarrow B$ be a Weierstrass fibration. The  \emph{fundamental line bundle} of $(\overline{S},\pi)$ is defined as 
\[\mathbb{L} := (R^1\pi_*\mathcal{O}_{\overline{S}})^{-1}.\]
\end{definition}

Using this theory, we are able to give a method by which elliptic fibrations can be explicitly constructed.

\begin{theorem}\textup{\cite[Sect. III.1]{btes}} Let $(\overline{S},\pi)$ be a Weierstrass fibration over a smooth curve $B$. Then
\begin{enumerate}
\item There is an embedding,
\[f\colon \overline{S} \hookrightarrow \mathbb{P}(\mathcal{O}_B \oplus \mathbb{L}^{-2} \oplus \mathbb{L}^{-3}).\]
\item If $p\colon \mathbb{P}(\mathcal{O}_B \oplus \mathbb{L}^{-2} \oplus \mathbb{L}^{-3}) \rightarrow B$ is the natural projection map, then $p \circ f = \pi$. 
\item The hypersurface $\overline{S}$ is given by the vanishing of a section of $\mathcal{O}_\mathbb{P}(1)^3 \otimes p^*\mathbb{L}^6$ where $\mathcal{O}_\mathbb{P}(1)$ is the inverse of the tautological bundle on the projective bundle $\mathbb{P}(\mathcal{O}_B \oplus \mathbb{L}^{-2} \oplus \mathbb{L}^{-3})$.
\item $\overline{S}$ can be written as the vanishing locus of 
\[ZY^2 = X^3 + \alpha XZ^2 + \beta Z^3\]
where $\alpha$ and $\beta$ are global sections of $p^*\mathbb{L}^4$ and $p^*\mathbb{L}^6$, and $X,Y,Z$ are global sections of $\mathcal{O}_\mathbb{P}(1) \otimes p^*\mathbb{L}^2$, $\mathcal{O}_\mathbb{P}(1) \otimes p^*\mathbb{L}^3$ and $\mathcal{O}_\mathbb{P}(1)$ respectively.
\end{enumerate}
\end{theorem}

\begin{remark} We will often refer to $\alpha$ and $\beta$ as sections of $\mathbb{L}^4$ and $\mathbb{L}^6$ respectively, using the fact that the spaces of sections of $\mathbb{L}^4$ and $\mathbb{L}^6$ are isomorphic to the spaces of sections of their pull-backs under $p$.
\end{remark}

This explains the meaning of the name ``Weierstrass fibration'': such surfaces admit expressions which are completely analogous to the Weierstrass form of an elliptic curve over a number field. This may be viewed as a refined version of the fact that the generic fibre of $(\overline{S},\pi)$ is an elliptic curve over the function field $\mathbb{C}(B)$, which may itself be expressed in Weierstrass form (see, for instance, \cite[Chap. II]{btes}).

Our next aim is to find conditions under which this construction gives a K3 surface. We already know that  $h^1(S,\mathcal{O}_S) = 0$ if and only if the base curve $B$ is $\Proj^1$, so it just remains to compute the canonical bundle.

It is easy to compute the dualizing sheaf of a Weierstrass fibration. If $\omega_B$ is the canonical bundle on the curve $B$, then 
\[\omega_{\mathbb{P}(\mathcal{O}_B \oplus \mathbb{L}^{-2} \oplus \mathbb{L}^{-3})} \cong p^*(\omega_B \otimes \mathbb{L}^{-5}) \otimes \calO_{\Proj}(1)^{-3}. \]
Using this, adjunction gives:

\begin{proposition}\textup{\cite[Prop. III.1.1]{btes}} The dualizing sheaf of $\overline{S}$ is given by
\[\omega_{\overline{S}} = \pi^*(\omega_B \otimes \mathbb{L}).\]
\end{proposition}

If $(S,\pi)$ is an elliptically fibred K3 surface and $(\overline{S},\pi)$ is the Weierstrass model associated to $(S,\pi)$, then $\omega_{\overline{S}} \cong \calO_{\overline{S}}$. Therefore, if $S$ is a K3 surface, then $\mathbb{L} \cong \omega_{\Proj^1}^{-1} \cong \mathcal{O}_{\mathbb{P}^1}(2)$. The converse is also true: if $\mathbb{L} \cong \mathcal{O}_{\mathbb{P}^1}(2)$ then the minimal resolution $S$ of the Weierstrass model $\overline{S}$ is a K3 surface (the proof is easy, but relies upon the fact that the exceptional curves in the resolution of an ADE singularity do not contribute to the canonical bundle; in technical language we say that ADE singularities admit \emph{crepant} resolutions). Thus we find:

\begin{proposition} Let $\pi\colon S \to B$ be an elliptic fibration, $(\overline{S},\pi)$ be its Weierstrass model, and $\mathbb{L} := (R^1\pi_*\mathcal{O}_{\overline{S}})^{-1}$ be its fundamental line bundle. Then $S$ is a K3 surface if and only if $B = \Proj^1$ and $\mathbb{L} \cong \calO_{\Proj^1}(2)$.
\end{proposition}

Using this, we may express any elliptically fibred K3 surface as a hypersurface in $\mathbb{P}(\mathcal{O}_{\mathbb{P}^1}  \oplus\mathcal{O}_{\mathbb{P}^1}(-4) \oplus\mathcal{O}_{\mathbb{P}^1}(-6))$ given by an equation
\[ZY^2 = X^3 + \alpha(s,t)XZ^2 + \beta(s,t)Z^3\]
where $(s,t)$ are coordinates on $\Proj^1$ and $\alpha(s,t)$ and $\beta(s,t)$ are homogeneous polynomials in $s$ and $t$ of degrees $8$ and $12$ respectively.

\begin{remark} The projective bundle $\mathbb{P}(\mathcal{O}_{\mathbb{P}^1}  \oplus\mathcal{O}_{\mathbb{P}^1}(-4) \oplus\mathcal{O}_{\mathbb{P}^1}(-6))$ admits a birational contraction map to the weighted projective space $\mathbb{WP}(1,1,4,6)$. Under this contraction, elliptic K3 surfaces are expressed as weighted projective hypersurfaces of the form
\[y^2  = x^3 + \alpha(s,t)x + \beta(s,t)\]
with $\alpha(s,t)$ and $\beta(s,t)$ exactly as above. Here $x$ has weight $4$, $y$ has weight $6$, and $s$, and $t$ both have weight $1$.
\end{remark}

\begin{remark}
\label{remark:Hpol}
A general K3 surface with elliptic fibration is a smooth hypersurface in $\mathbb{P}(\mathcal{O}_{\mathbb{P}^1}  \oplus\mathcal{O}_{\mathbb{P}^1}(-4) \oplus\mathcal{O}_{\mathbb{P}^1}(-6))$, which has Picard rank $2$. By the technique of Sect. \ref{section:embeddings}, one may show that this embedding induces a pseudo-ample $H$-polarization on these K3 surfaces. This fits with our observation from Sect. \ref{sect:Hpol}, that elliptically fibred K3 surfaces all admit $H$-polarizations.
\end{remark}

\subsection{Singular Fibres Revisited}
\label{section:fibresagain}

Now that we know how to construct elliptically fibred K3 surfaces, our next task is to find lattice polarizations on them. In order to do this, we will need to perform a closer study of the N\'{e}ron-Severi lattice of an elliptically fibred surface. We begin with a closer examination of the singular fibres.

Assume that we begin with a (possibly singular) elliptically fibred surface $(\overline{S},\pi)$ in Weierstrass form. We would like to use the local behaviour of the Weierstrass equation of $(\overline{S},\pi)$ to describe the configurations of divisors arising from resolution of the singularities of $\overline{S}$. 

First of all, it is clear that a fibre of $(\overline{S},\pi)$ is a singular elliptic curve if and only if the discriminant of the cubic
\[ZY^2 = X^3 + \alpha XZ^2 +  \beta Z^3\]
vanishes. As usual, we may express this discriminant as a polynomial in $\alpha$ and $\beta$, giving
\[\Delta = 4\alpha^3 + 27\beta^2 \in H^0(B,\mathbb{L}^{12}).\]
The points at which $\Delta$ vanishes correspond to the discriminant locus of the fibration $(\overline{S},\pi)$. 

Kodaira \cite{casII} and Tate \cite{adtsfep} showed how to use the local behaviour of $\alpha,\beta$ and $\Delta$ to detect singularities in the surface $\overline{S}$ and computed the minimal resolutions of these singularities.

\begin{theorem}\textup{\cite[Sect. IV.3]{btes}} \label{theorem:Tate}
Let $(S,\pi)$ be a smooth relatively minimal elliptically fibred surface and let $(\overline{S},\pi)$ be its Weierstrass model. Let $\mathbb{L}$ be the fundamental line bundle of $(\overline{S},\pi)$ and let $\alpha \in H^0(B,\mathbb{L}^4)$ and $\beta \in H^0(B,\mathbb{L}^6)$ be the sections defining $\overline{S}$. Let $\Delta \in H^0(B,\mathbb{L}^{12})$ be the discrimant.

Denote by $\nu_p(\alpha)$ the order of vanishing of $\alpha$ at the point $p$, by $\nu_p(\beta)$ the order of vanishing of $\beta$ at $p$ and by $\nu_p(\Delta)$ the order of vanishing of $\Delta$ at $p$. Then the fibre $\pi^{-1}(p)$ is singular if and only if $\Delta(p) = 0$. The singularity of $\overline{S}$ lying over $p$, along with a description of the singular fibre $\pi^{-1}(p)$ in $S$ and its Euler characteristic $e$ are given by Table \ref{table:singfib}.
\end{theorem}

\begin{table}
\caption{Singular fibres of a smooth relatively minimal elliptically fibred surface}
\small{
    \begin{tabular}{ccccccp{4cm}}
\hline\noalign{\smallskip}
Name & $\nu_p(\alpha)$ & $\nu_p(\beta)$  & $\nu_p(\Delta)$ & Singularity in $\overline{S}$ & $e$ & Description \\
\noalign{\smallskip}\hline\hline\noalign{\smallskip}
$I_1$ & $0$ & $0$ & $1$ & Smooth & $1$ & Nodal rational curve\\ \hline\noalign{\smallskip}
$I_n$ & $0$ & $0$ & $n$ & $A_{n-1}$ & $n$ &  Cycle of $n$ smooth rational curves with dual graph $\widetilde{A}_{n-1}$ \\ \hline\noalign{\smallskip}
$I_0^*$ & $2$ & $3$ & $6$ & $D_4$ & $6$ & Configuration of $5$ smooth\\ 
             & $\geq 3$ & $3$ & $6$ &   &    &    rational curves with dual  \\ 
             & $2$ & $\geq 4$  & $6$   &  &   &        graph $\widetilde{D}_4$  \\ \hline\noalign{\smallskip}
$I_n^*$ & $2$ & $3$ & $n+6$ & $D_{4+n}$ & $n+6$ & Configuration of $n+5$ smooth rational curves with dual graph $\widetilde{D}_{n+4}$ \\  \hline\noalign{\smallskip}
$II$ &$\geq 1$ & $1$ & $2$ & Smooth & $2$ & Cuspidal rational curve \\ \hline\noalign{\smallskip}
$III$ & $1$ & $\geq 2$ & $3$ & $A_1$ & $3$ &Two rational curves tangent at a point \\  \hline\noalign{\smallskip}
$IV$ & $\geq 2$ & $2$ & $4$ &  $A_2$ & $4$ & Three smooth rational curves meeting at a single point \\ \hline\noalign{\smallskip}
$IV^*$ &$\geq 3$ & $4$ & $8$ & $E_6$ & $8$ & Configuration of $7$ smooth rational curves with dual graph $\widetilde{E}_6$ \\  \hline\noalign{\smallskip}
$III^*$ &$3$ & $\geq 5$ & $9$ & $E_7$ & $9$ & Configuration of $8$ smooth rational curves with dual graph $\widetilde{E}_7$ \\  \hline\noalign{\smallskip}
$II^*$ &$\geq 4$ & $5$ & $10$ & $E_8$ & $10$ & Configuration of $9$ smooth rational curves with dual graph $\widetilde{E}_8$ \\  \hline\noalign{\smallskip}
    \end{tabular}
}
\label{table:singfib}
\end{table}

\begin{remark}If the values of $\nu_p(\alpha),\nu_p(\beta)$ and $\nu_p(\Delta)$ do not fall into one of the classes described in Table \ref{table:singfib}, then the singularities of $\overline{S}$ are worse than ADE singularities and $\overline{S}$ is not the Weierstrass model of a smooth relatively minimal elliptically fibred surface.\end{remark}

\subsection{Mordell-Weil Group}

Theorem \ref{theorem:Tate} enables us to compute the classes in the N\'{e}ron-Severi lattice that arise from components of fibres of the elliptic fibration. However, to obtain the full N\'{e}ron-Severi lattice, we also have to know about the classes coming from sections. This data is encoded by a second object, the \emph{Mordell-Weil group}.

As usual, we let $(S,\pi)$ an elliptically fibred surface. Choose an arbitrary section $\mathbb{O}$ and, as before, let its image be denoted $O$.

Let $\eta$ be the generic point on the base $B$. Any section $Q$ determines a $\mathbb{C}(B)$-rational point $\widetilde{Q}$ on the generic fibre $S_\eta$ and, in fact, there is a bijective correspondence between sections of $(S,\pi)$ and $\mathbb{C}(B)$-rational points of $S_\eta$. To see this, note that every $\mathbb{C}(B)$-rational point $\widetilde{Q}$ is actually a $\mathbb{C}[B_0]$ point of $S|_{B_0}$ over some open set $B_0 \subseteq B$; this bijection associates to $\tilde{Q}$ the closure in $S$ of the corresponding point in $\mathbb{C}[B_0]$. 

The group structure on $S_\eta$ allows us to add two sections ${Q}_1$ and ${Q}_2$, by letting $Q_1 + Q_2$ be the closure of  the $\mathbb{C}(B)$-rational point $\widetilde{Q}_1 \oplus \widetilde{Q}_2$ (where $\oplus$ is used to indicate addition in the group structure on $S_\eta$, defined with respect to the zero section $\widetilde{O}$).

\begin{definition}
Let $(S,\pi)$ be an elliptically fibred surface with a chosen zero section $O$. Then the set of sections of $(S,\pi)$, equipped with the group structure defined above, is called the \emph{Mordell-Weil group} of $(S,\pi)$ and denoted $\MW(S,\pi)$.
\end{definition}

By this prescription, the group structure on $\MW(S,\pi)$ corresponds directly to the pointwise addition in each smooth fibre. Note that, as usual, the structure of $\MW(S,\pi)$ does not depend upon the choice of section $O$, hence our notation for the Mordell-Weil group does not make reference to it.

\subsection{The N\'{e}ron-Severi Lattice}

Now we determine the relationship between the N\'eron-Severi lattice of a smooth relatively minimal elliptically fibred surface $(S,\pi)$, its Mordell-Weil group, and its singular fibres. 

Let $O$ be a designated section in $\MW(S,\pi)$, which we will use as the identity element. To each fibre $\pi^{-1}(p)$ of $(S,\pi)$ we may associate a root lattice $R_p$ as follows: write $\pi^{-1}(p)$ as a sum of irreducible components $C_0,\ldots,C_n$ as before, labelled so that $C_0$ is the component which intersects the zero section $O$, then let $R_p$ be the sublattice of $\NS(S)$ generated by the classes of $C_1, \dots , C_n$. Note that if $\pi^{-1}(p)$ is irreducible, then $R_p$ is trivial. Next define a sublattice $L_\pi$ of $\NS(S)$ to be the sublattice spanned by the classes $[E], [O]$ and the lattices $R_p$ for all points $p \in B$.

\begin{theorem}\textup{\cite[Thm. VII.2.1]{btes}}
\label{theorem:NSEF}
There is an exact sequence of abelian groups,
\[0 \longrightarrow L_\pi \stackrel{a}{\longrightarrow} \NS(S) \stackrel{b}{\longrightarrow} \MW(S,\pi) \longrightarrow 0,\]
where $a$ is the obvious embedding and $b$ is the composition of the restriction to the generic fibre $\NS(S) \to \Pic(S_{\eta})$ with the homomorphism $\Pic(S_{\eta}) \to \MW(S,\pi)$. In particular, $b$ assigns to a section $Q$ of $(S,\pi)$ the associated class in $\MW(S,\pi)$.
\end{theorem}

Note that the homomorphisms $a$ and $b$ depend upon the choice of section $O$, since the definition of $L_\pi$ depends upon the choice of $O$.

Therefore, to determine the N\'eron-Severi lattice of any elliptic fibration, it is enough to know both the structure of the singular fibres of $(S,\pi)$ and the Mordell-Weil group. Note that Thm. \ref{theorem:NSEF} implies that $\MW(S,\pi)$ is necessarily finitely generated.

As we have seen, the singular fibres of an elliptic surface are quite easy to determine, given an explicit Weierstrass equation for $(S,\pi)$. In general, however, it is quite difficult to compute the non-torsion part of the Mordell-Weil group (see \cite[Chap. VII]{btes}), but we can often obtain bounds on the size of $\MW(S,\pi)_{\mathrm{tors}}$, the torsion part of $\MW(S,\pi)$.

\begin{corollary}
\label{corollary:torsion}
The group $\MW(S,\pi)_{\mathrm{tors}}$ is isomorphic to $ (L_\pi\otimes \mathbb{Q} \cap \NS(S)) / L_\pi$ and hence is a subgroup of the discriminant group of $L_\pi$.
\end{corollary}
\begin{proof}
The first statement follows directly from Thm. \ref{theorem:NSEF}. Note that $ (L_\pi \otimes \mathbb{Q} \cap \NS(S)) \subseteq L_\pi^*$, hence the second claim follows.
\end{proof}

\begin{corollary}
\label{corollary:Ellipticpolarization}
Let $S$ be a K3 surface and let $\pi$ be an elliptic fibration on $S$. Let $N_\pi$ be the sublattice of $\NS(S)$ generated by $L_\pi$ and $\MW(S,\pi)_{\text{tors}}$. Then $S$ is $N_\pi$-polarized.
\end{corollary}
\begin{proof}
This follows directly from Thm. \ref{theorem:NSEF}. Since $\mathbb{Z}[O] \oplus \mathbb{Z}[E] \cong H$ is contained in $L_\pi$, the discussion in Sect. \ref{sect:Hpol} shows that $N_\pi$ contains a pseudo-ample class, so condition 2 of Def. \ref{definition:lattice-polarization} is satisfied. To see that condition 1 is also satisfied, we need to show that $N_\pi$ is a primitive sublattice of $\NS(S)$. But $\NS(S)/L_\pi \cong \MW(S,\pi)$ and hence, by construction, $\NS(S)/N_\pi \cong \MW(S,\pi)_{\mathrm{free}}$, where $\MW(S,\pi)_{\mathrm{free}}$ denotes the torsion-free part of $\MW(S,\pi)$. Thus $N_\pi$ is primitively embedded in $\NS(S)$. 
\end{proof}

\subsection{Examples}

Now we will bring all of this theory together to examine some examples of K3 surfaces which are elliptically fibred, explaining how to use the techniques in the previous sections to compute lattice polarizations on them.

\begin{example}[Polarization by a lattice of rank 18]
Let us take the elliptically fibred K3 surface $(S,\pi)$ obtained as the minimal resolution of the Weierstrass fibration $(\overline{S},\pi)$ given by
\[Y^2Z  = X^3 + s^4t^4 XZ^2 + s^5t^5(as^2 + bst + ct^2)Z^3,\]
where $a,b,c$ are parameters chosen generically in $\mathbb{C}$. According to Thm. \ref{theorem:Tate}, $(S,\pi)$ has two singular fibres of type $II^*$ located at $[s:t] = [0:1]$ and $[1:0]$. One may check that for a general choice of $a,b,c$, the discriminant $\Delta$ vanishes simply at four other points in $\mathbb{P}^1$, giving four further singular fibres of type $I_1$. Using Thm. \ref{thm:euler} and Table \ref{table:singfib}, we see that the Euler characteristic of this surface is $2(10) + 4(1) = 24$, as expected.

For a singular fibre of type $II^*$, the root lattice $R_p$ is isomorphic to $(-E_8)$. Therefore, the lattice $L_\pi$ for this K3 surface is isomorphic to $H \oplus (-E_8) \oplus (-E_8)$ (recall that singular fibres of type $I_1$ are irreducible, so do not contribute to $L_{\pi}$). Since $L_\pi$ is unimodular, Cor. \ref{corollary:torsion} shows that $\MW(S,\pi)_{\mathrm{tors}}$ is trivial, so Cor. \ref{corollary:Ellipticpolarization} shows that $S$ is in fact $H \oplus (-E_8) \oplus (-E_8)$ polarized. This example is explored in great detail in the papers \cite{milpk3s} and \cite{desk3sni}.
\end{example}

\begin{example}[Polarization by a lattice of rank 19]
Now take the elliptically fibred K3 surface $(S,\pi)$ obtained as the minimal resolution of the Weierstrass fibration $(\overline{S},\pi)$ given by
\begin{eqnarray*}Y^2 & = &\ X^3  + \frac{1}{3}t^{3}s^3  (48 lt^{2} + 96 l ts -  t + 48 l s^2 )XZ^2- \\
 & & - \frac{2}{27}t^{5} s^5 (72 t^{2} l + 144 lts  -  ts + 72ls^2 )Z^3,\end{eqnarray*}
where $l$ is a generic parameter in $\mathbb{C}$. Its discriminant is given by
\[\Delta(s,t) = 256 l^{2}  (t + 1)^{4} s^9 t^{9} (64 l t^{2}  + 128 lts  -  ts + 64 ls^2).\]
By Thm. \ref{theorem:Tate}, for generic $l$, this K3 surface has two singular fibres of type $III^*$ occurring at $[s:t] = [0:1]$ and $[1:0]$, one singular fibre of type $I_4$ at $[1:-1]$, and two $I_1$'s occurring at the zeros of $64 l t^{2}  + 128 lts  -  ts + 64 ls^2$. 

We thus find that $L_\pi = H \oplus (-E_7) \oplus (-E_7) \oplus (-A_3)$.  By Ex. \ref{example:HE7E7} in the appendix, the lattice $L_\pi$ has a unique overlattice of index $2$ isomorphic to $H \oplus (-E_8) \oplus (-E_8) \oplus \langle -4\rangle$. Therefore, by Cor. \ref{corollary:torsion}, the only possible torsion in $\MW(S,\pi)$ is of order two and, if $\MW(S,\pi)_{\text{tors}}$ has order two, then Cor. \ref{corollary:Ellipticpolarization} shows that $S$ is $H \oplus (-E_8) \oplus (-E_8) \oplus \langle -4\rangle$-polarized. 

In fact, one can check that $X = t^2s^2/3$ is a solution to the right hand side of the Weierstrass equation for $(\overline{S},\pi)$. Thus 
\[[s:t] \mapsto [X:Y:Z] = \left[\frac{t^2s^2}{3} : 0 : 1 \right]\]
is a section of $\pi$. Since $Y=0$, we see easily that this is actually an order $2$ torsion section of $\pi$ on each fibre, hence it has order two in $\MW(S,\pi)$. This shows that $S$ is $H \oplus (-E_8) \oplus (-E_8) \oplus \langle -4\rangle$-polarized.
\end{example}

\section{Ample and K\"ahler Cones}

In the final section of these notes we will discuss the ample and K\"ahler cones of a K3 surface. These are important objects: among other things, their geometry controls fibration structures and automorphisms on the K3 surface. Furthermore, we will find that the description of these cones essentially reduces to lattice theory. This should not come as a surprise: after all, we have already seen that the Torelli Theorems reduce the theory of moduli of lattice polarized K3 surfaces to essentially lattice theoretic considerations, so it does not seem unreasonable to expect that the birational geometry of K3 surfaces might also be lattice theoretic in nature. 

\subsection{The Ample Cone}\label{sect:amplecone}

Begin by letting $S$ be any smooth projective complex surface. Then we have:

\begin{definition} The \emph{ample cone} of $S$ is the set $\mathrm{Amp}(S) \subset \NS(S) \otimes \R$ consisting of finite sums $\sum a_iu_i$, with $u_i \in \NS(S)$ ample and $a_i \in \R_{>0}$.
\end{definition}

By the Nakai-Moishezon ampleness criterion, a class $u$ in $\NS(S) \otimes \mathbb{R}$ is in $\Amp(S)$ if and only if it satisfies $\langle u,u\rangle > 0$ and $\langle u,[C] \rangle >0$ for every irreducible curve $C$ on $S$. In the case where $S$ is a K3 surface, we will see that it suffices to check that $\langle u, [C] \rangle > 0 $ for every smooth rational curve $C$ on $S$. 

To state this formally, we introduce some notation. The set
\[\NS(S)^+ = \{ u \in \NS(S) \otimes \mathbb{R} \mid \langle u,u\rangle > 0 \}\]
consists of two disjoint connected cones. All of the ample classes in $\NS(S)$ belong to one of them, the \emph{positive cone}, which we denote by $\calC_S$.  The ample cone may then be described by the following theorem, which is an easy consequence of the Nakai-Moishezon criterion.

\begin{theorem}
\label{theorem:Kahler}
Let $S$ be a projective K3 surface and let $\Delta^+(S)$ be the set of classes in $\NS(S)$ which are represented by smooth rational curves on $S$. Then the ample cone of $S$ is given by the intersection between the positive cone $\calC_S$ and the set
\[\{u \in \NS(S) \otimes \mathbb{R} \mid \langle u,\delta \rangle > 0\  \mathrm{for}\  \mathrm{all}\ \delta \in \Delta^+(S) \}.\]
\end{theorem}
\begin{proof} By the Nakai-Moishezon criterion, $u \in \Amp(S)$ if and only if $\langle u, u \rangle > 0$ and $\langle u,[C] \rangle > 0$ for every irreducible curve $C$ on $S$. 

Suppose first that $C$ satisfies $\langle [C],[C] \rangle \geq 0$. Then $\langle h, [C] \rangle > 0$ for any ample class $h$ in $\NS(S)$ so, by \cite[Cor. IV.7.2]{bpv}, we must have $[C] \in \overline{\calC}_S$ (the closure of $\calC_S$). Applying \cite[Cor. IV.7.2]{bpv} again, we find that $\langle u, u \rangle > 0$ and $\langle u,[C] \rangle > 0$ if and only if $u \in \calC_S$.

Thus $u \in \Amp(S)$ if and only if $u \in \calC_S$ and $\langle u,[C] \rangle > 0$ for every irreducible curve $C$ on $S$ with $\langle [C],[C] \rangle <0$. But, by the genus formula \cite[Ex. V.1.3]{hart}, the irreducible curves $C$ on $S$ with $\langle [C],[C] \rangle <0$ are precisely the smooth rational curves on $S$.
\end{proof}

Next, we define three special subgroups of the orthogonal group $\mathrm{O}(\NS(S))$. To define the first, note that any isometry in $\mathrm{O}(\NS(S))$ must either preserve or exchange the two components of $\NS(S)^+$. Let $\mathrm{O}^+(\NS(S))$ denote the subgroup of isometries that preserve them.

To define the second, suppose that $\delta$ is any element of $\NS(S)$ with $\langle \delta,\delta \rangle = -2$. We can define an isometry of $\NS(S)$ by $u \mapsto u + \langle u, \delta\rangle \delta$. Such an isometry is called a \emph{Picard-Lefschetz reflection}. The subgroup of $\mathrm{O}(\NS(S))$ generated by all Picard-Lefschetz reflections is called the \emph{Weyl group} of the lattice $\NS(S)$ and is denoted $W_S$. It is easy to see that Picard-Lefschetz reflections preserve the positive cone $\calC_S$, so $W_S$ is a subgroup of $\mathrm{O}^+(\NS(S))$.  

In fact more is true. One may show (see \cite[Sect. 4.2]{iqfaag}) that $W_S$ is a normal subgroup of $\mathrm{O}^+(\NS(S))$ and that there is a third group $G_S$ giving a semidirect product decomposition
\[\mathrm{O}^+(\NS(S)) \cong W_S \rtimes G_S.\]
Furthermore, the discussion in  \cite[Sect. 4.2]{iqfaag} shows that the closure of the cone $\Amp(S)$ is a fundamental domain for the action of $W_S$ on the positive cone $\calC_S$, so $G_S$ should be thought of as the group of symmetries of the ample cone of $S$.

\begin{example}[Elliptically fibred K3 surface]
Let $S$ be a K3 surface with $\NS(S) \cong H$ (recall that, by Sect. \ref{sect:Hpol}, this implies that $S$ is projective and admits an elliptic fibration). We will take a basis of $\NS(S)$ given by $[E]$ and $[O]$, the classes of a fibre and the section of the elliptic fibration $S \rightarrow \mathbb{P}^1$ respectively.

We may calculate easily that there are only two classes $\delta$ in $H$ with $\langle \delta, \delta \rangle =-2$, given by $\delta = \pm [O]$. Clearly $-[O]$ is not the class of an effective divisor, so the only smooth rational curve on $S$ is $O$ itself. The set $\NS(S)^+$ is given by 
\[\langle a[E] + b[O],a[E]+b[O]\rangle = 2ab - 2b^2 > 0\]
and the class $3[E] + [O]$ is ample (by Nakai-Moishezon), so the positive cone $\calC_S$ is determined by the conditions $b > 0$ and $a > b$. To find $\Amp(S)$, we also require the condition that 
\[ \langle a[E]+b[O], [O] \rangle = a-2b > 0.\]
Therefore, we conclude that $\Amp(S)$ is given by the open cone in $\mathbb{R}^2$ defined by $b > 0$ and $a > 2b$. The ample and positive cones for this example are displayed in Fig. \ref{ampleconefig1}; in this figure the positive cone $\calC_S$ is the entire shaded area, whilst the darker shaded subset of $\calC_S$ is the ample cone $\Amp(S)$.

%\begin{figure}[t]
%\sidecaption[t]
%\setlength{\unitlength}{1cm}
%\begin{picture}(6.5,4.5)
%\put(0,0){\vector(0,1){4.5}}
%\put(0,0){\vector(1,0){4.5}}
%\put(0,0){\vector(1,1){4.5}}
%\put(0,0){\vector(2,1){4.5}}
%\put(0,0){\circle*{0.1}}
%\put(0,1){\circle*{0.1}}
%\put(0,2){\circle*{0.1}}
%\put(0,3){\circle*{0.1}}
%\put(0,4){\circle*{0.1}}
%\put(1,0){\circle*{0.1}}
%\put(1,1){\circle*{0.1}}
%\put(1,2){\circle*{0.1}}
%\put(1,3){\circle*{0.1}}
%\put(1,4){\circle*{0.1}}
%\put(2,0){\circle*{0.1}}
%\put(2,1){\circle*{0.1}}
%\put(2,2){\circle*{0.1}}
%\put(2,3){\circle*{0.1}}
%\put(2,4){\circle*{0.1}}
%\put(3,0){\circle*{0.1}}
%\put(3,1){\circle*{0.1}}
%\put(3,2){\circle*{0.1}}
%\put(3,3){\circle*{0.1}}
%\put(3,4){\circle*{0.1}}
%\put(4,0){\circle*{0.1}}
%\put(4,1){\circle*{0.1}}
%\put(4,2){\circle*{0.1}}
%\put(4,3){\circle*{0.1}}
%\put(4,4){\circle*{0.1}}
%\put(0.1,4.25){$[O]$}
%\put(4.6,0){$[E]$}
%\put(4.6,2.15){$2[E]+[O]$}
%\put(4.6,4.25){$[E]+[O]$}
%\put(3,0.6){$\Amp(S)$}
%\put(3,2.3){$\calC_S$}
%\end{picture}
%\caption{Ample and positive cones for a K3 surface with $\NS(S) \cong H$. Note that $\Amp(S) \subset \calC_S$}
%\label{ampleconefig1}
%\end{figure}

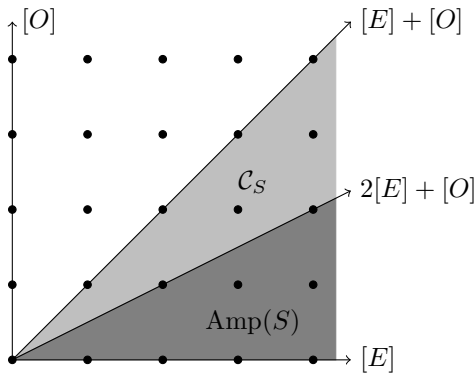
\begin{figure}[t]
\begin{tikzpicture}
\draw [lightgray, fill=lightgray] (0,0)--(4.3,2.15)--(4.3,4.3)--(0,0);
\draw [gray, fill=gray] (0,0)--(4.3,2.15)--(4.3,0)--(0,0);
\draw [->] (0,0)--(0,4.5);
\draw [->] (0,0)--(4.5,0);
\draw [->] (0,0)--(4.5,4.5);
\draw [->] (0,0)--(4.5,2.25);
\draw [fill] (0,0) circle [radius=0.05]; 
\draw [fill] (0,1) circle [radius=0.05]; 
\draw [fill] (0,2) circle [radius=0.05]; 
\draw [fill] (0,3) circle [radius=0.05]; 
\draw [fill] (0,4) circle [radius=0.05]; 
\draw [fill] (1,0) circle [radius=0.05]; 
\draw [fill] (1,1) circle [radius=0.05]; 
\draw [fill] (1,2) circle [radius=0.05]; 
\draw [fill] (1,3) circle [radius=0.05]; 
\draw [fill] (1,4) circle [radius=0.05]; 
\draw [fill] (2,0) circle [radius=0.05]; 
\draw [fill] (2,1) circle [radius=0.05]; 
\draw [fill] (2,2) circle [radius=0.05]; 
\draw [fill] (2,3) circle [radius=0.05]; 
\draw [fill] (2,4) circle [radius=0.05]; 
\draw [fill] (3,0) circle [radius=0.05]; 
\draw [fill] (3,1) circle [radius=0.05]; 
\draw [fill] (3,2) circle [radius=0.05]; 
\draw [fill] (3,3) circle [radius=0.05]; 
\draw [fill] (3,4) circle [radius=0.05]; 
\draw [fill] (4,0) circle [radius=0.05]; 
\draw [fill] (4,1) circle [radius=0.05]; 
\draw [fill] (4,2) circle [radius=0.05]; 
\draw [fill] (4,3) circle [radius=0.05]; 
\draw [fill] (4,4) circle [radius=0.05]; 
\node [right] at (0,4.5) {$[O]$};
\node [right] at (4.5,0) {$[E]$};
\node [right] at (4.5,2.25) {$2[E]+[O]$};
\node [right] at (4.5,4.5) {$[E] + [O]$};
\node [below] at (3.2,0.8) {$\Amp(S)$};
\node [above] at (3.2,2.1) {$\calC_S$};
\end{tikzpicture}
\caption{Ample and positive cones for a K3 surface with $\NS(S) \cong H$. The positive cone $\calC_S$ is the entire shaded area, whilst the darker shaded subset of $\calC_S$ is the ample cone $\Amp(S)$}
\label{ampleconefig1}
\end{figure}

Note that the closure of the cone $\Amp(S)$ is rational polyhedral, hence $G_S$ must be finite. We will see in Sect. \ref{section:automorphisms} that this means $S$ has a finite group of automorphisms.
\end{example}

\begin{example}[Nodal quartics]
\label{example:P1P2cone}
Let $S$ be the minimal resolution of a nodal quartic in $\mathbb{P}^3$ as discussed in Ex. \ref{example:nodalquartic}. In that example we found that $S$ is lattice polarized by the lattice $L$ with Gram matrix
\[\left(\begin{matrix} -2 & 0 \\ 0 & 4 \end{matrix} \right);\]
suppose now that $\NS(S) \cong L$.

As in Ex. \ref{example:nodalquartic}, let $H$ denote a divisor on $S$ with $H^2 = 4$ induced by a general hyperplane section in $\Proj^3$ and let $C$ denote the exceptional $(-2)$-curve arising from the blow-up of the node. Then $\NS(S)$ is generated by the classes $[H]$ and $[C]$ by assumption, so the set $\NS(S)^+$ is given by the condition
\[ \langle a[C] + b[H], a[C] + b[H] \rangle = -2a^2 + 4b^2 > 0. \] 
Since the class $[H] - [C]$ is ample, the positive cone $\calC_S$ is given by the inequalities $b > 0$ and $-\sqrt{2}b <  a < \sqrt{2}b$.

In order to compute the ample cone, we have to find the smooth rational curves on $S$. The divisor $C$ is one such curve. Another may be computed as follows. Note that projection away from the node in $\Proj^3$ induces a double covering map $S \to \Proj^2$, which maps $C$ isomorphically to a conic in $\Proj^2$. The involution exchanging the sheets of this cover maps $C$ to another smooth rational curve $C'$. As the pull-back of a line in $\Proj^2$ under this double cover is a divisor in the linear system $|H-C|$ on $S$, it is not difficult to show that $C'$ lies in the linear system $|2H-3C|$.

We claim that $C$ and $C'$ determine the ample cone of $S$. Consider the class  
\begin{align*}u_{\epsilon} &:= (1 + \epsilon)\left([H]-[C]\right) + \frac{1}{2}\left\langle [H]-[C],[C]\right\rangle [C] \\
&= (1 + \epsilon) [H] -\epsilon [C],\end{align*}
where $\epsilon \geq 0$ is a real number. Now $u_{\epsilon}$ satisfies $\langle u_{\epsilon} ,[C] \rangle = 2\epsilon$ and, for any other class $\delta \in \Delta^+(S)$, we have  $\langle u_{\epsilon} ,\delta \rangle \geq (1+\epsilon) \langle [H]-[C],\delta \rangle > 0$, since $[H]-[C]$ is ample and $\langle[C],\delta\rangle \geq 0$. So, by Thm. \ref{theorem:Kahler}, $u_{\epsilon} \in \Amp(S)$ for all $\epsilon > 0$, but $u_{0} = [H] \notin \Amp(S)$ as $\langle [H] ,[C] \rangle = 0$. Thus the class $[H]$ lies in the boundary of the ample cone. By a similar argument, one can show that the class
\begin{align*}u_0' &:= \left([H]-[C]\right) + \frac{1}{2}\left\langle [H]-[C],[C']\right\rangle [C']\\
& = 3[H] - 4[C] \end{align*}
also lies in the boundary of $\Amp(S)$. But $\rank \NS(S) = 2$ in this example, so the ample cone is $2$-dimensional and thus has only two boundary rays. 

The ample and positive cones for this example are shown in Fig. \ref{ampleconefig2}; as before, in this figure the positive cone $\calC_S$ is the entire shaded area, whilst the darker shaded subset of $\calC_S$ is the ample cone $\Amp(S)$. 

%\begin{figure}[t]
%\setlength{\unitlength}{1cm}
%\begin{picture}(11,3.5)
%\put(5.5,0){\vector(0,1){3.5}}
%\put(5.5,0){\vector(-1,0){3.5}}
%\put(5.5,0){\vector(1,0){3.5}}
%\put(5.5,0){\vector(-4,3){3.5}}
%\put(5.5,0){\vector(4,3){3.5}}
%\put(2.5,0){\circle*{0.1}}
%\put(2.5,1){\circle*{0.1}}
%\put(2.5,2){\circle*{0.1}}
%\put(2.5,3){\circle*{0.1}}
%\put(3.5,0){\circle*{0.1}}
%\put(3.5,1){\circle*{0.1}}
%\put(3.5,2){\circle*{0.1}}
%\put(3.5,3){\circle*{0.1}}
%\put(4.5,0){\circle*{0.1}}
%\put(4.5,1){\circle*{0.1}}
%\put(4.5,2){\circle*{0.1}}
%\put(4.5,3){\circle*{0.1}}
%\put(5.5,0){\circle*{0.1}}
%\put(5.5,1){\circle*{0.1}}
%\put(5.5,2){\circle*{0.1}}
%\put(5.5,3){\circle*{0.1}}
%\put(6.5,0){\circle*{0.1}}
%\put(6.5,1){\circle*{0.1}}
%\put(6.5,2){\circle*{0.1}}
%\put(6.5,3){\circle*{0.1}}
%\put(7.5,0){\circle*{0.1}}
%\put(7.5,1){\circle*{0.1}}
%\put(7.5,2){\circle*{0.1}}
%\put(7.5,3){\circle*{0.1}}
%\put(8.5,0){\circle*{0.1}}
%\put(8.5,1){\circle*{0.1}}
%\put(8.5,2){\circle*{0.1}}
%\put(8.5,3){\circle*{0.1}}
%\put(5.6,3.25){$[H]$}
%\put(9.1,0){$[C]$}
%\put(1.3,0){$-[C]$}
%\put(9.1,2.6){$\sqrt{2}[C]+[H]$}
%\put(0.2,2.6){$-\sqrt{2}[C]+[H]$}
%\put(3.6,2.3){$\Amp(S)$}
%\put(6.7,2.3){$\calC_S$}
%\end{picture}
%\caption{Ample and positive cones for the K3 surface obtained as the minimal resolution of a nodal quartic in $\Proj^3$.  The positive cone $\calC_S$ is the entire shaded area, whilst the darker shaded subset of $\calC_S$ is the ample cone $\Amp(S)$}
%\label{ampleconefig2}
%\end{figure}

\begin{figure}[t]
\begin{tikzpicture}
\draw [lightgray, fill=lightgray] (0,0)--(3.3,2.33)--(3.3,3.3)--(-3.3,3.3)--(-3.3,2.33)--(0,0);
\draw [gray, fill=gray] (0,0)--(-3.3,2.475)--(-3.3,3.3)--(0,3.3)--(0,0);
\draw [->] (0,0)--(3.5,0);
\draw [->] (0,0)--(-3.5,0);
\draw [->] (0,0)--(0,3.5);
\draw [->] (0,0)--(3.5,2.47);
\draw [->] (0,0)--(-3.5,2.47);
\draw [->] (0,0)--(-3.5,2.625);
\draw [fill] (0,0) circle [radius=0.05]; 
\draw [fill] (0,1) circle [radius=0.05]; 
\draw [fill] (0,2) circle [radius=0.05]; 
\draw [fill] (0,3) circle [radius=0.05];  
\draw [fill] (1,0) circle [radius=0.05]; 
\draw [fill] (1,1) circle [radius=0.05]; 
\draw [fill] (1,2) circle [radius=0.05]; 
\draw [fill] (1,3) circle [radius=0.05]; 
\draw [fill] (2,0) circle [radius=0.05]; 
\draw [fill] (2,1) circle [radius=0.05]; 
\draw [fill] (2,2) circle [radius=0.05]; 
\draw [fill] (2,3) circle [radius=0.05];  
\draw [fill] (3,0) circle [radius=0.05]; 
\draw [fill] (3,1) circle [radius=0.05]; 
\draw [fill] (3,2) circle [radius=0.05]; 
\draw [fill] (3,3) circle [radius=0.05];  
\draw [fill] (-1,0) circle [radius=0.05]; 
\draw [fill] (-1,1) circle [radius=0.05]; 
\draw [fill] (-1,2) circle [radius=0.05]; 
\draw [fill] (-1,3) circle [radius=0.05];  
\draw [fill] (-2,0) circle [radius=0.05]; 
\draw [fill] (-2,1) circle [radius=0.05]; 
\draw [fill] (-2,2) circle [radius=0.05]; 
\draw [fill] (-2,3) circle [radius=0.05]; 
\draw [fill] (-3,0) circle [radius=0.05]; 
\draw [fill] (-3,1) circle [radius=0.05]; 
\draw [fill] (-3,2) circle [radius=0.05]; 
\draw [fill] (-3,3) circle [radius=0.05];  
\node [right] at (0,3.5) {$[H]$};
\node [right] at (3.5,0) {$[C]$};
\node [left] at (-3.5,0) {$-[C]$};
\node [right] at (3.5,2.47) {$\sqrt{2}[C]+[H]$};
\node [left] at (-3.5,2.3) {$-\sqrt{2}[C]+[H]$};
\node [left] at (-3.5,2.8) {$-4[C]+3[H]$};
\node [above] at (-1.5,2.1) {$\Amp(S)$};
\node [above] at (1.5,2.1) {$\calC_S$};
\end{tikzpicture}
\caption{Ample and positive cones for the K3 surface obtained as the minimal resolution of a nodal quartic in $\Proj^3$.  The positive cone $\calC_S$ is the entire shaded area, whilst the darker shaded subset of $\calC_S$ is the ample cone $\Amp(S)$}
\label{ampleconefig2}
\end{figure}

As in the previous example, the closure of the cone $\Amp(S)$ is rational polyhedral, so $G_S$ is finite and $S$ has a finite group of automorphisms. However, the structure of the Weyl group $W_S$ in this case is more interesting. $W_S$ is an infinite group, generated by the Picard-Lefschetz reflections corresponding to the classes $[C]$ and $[C']$. Repeated iteration of these reflections takes the boundary rays $[H]$ and $-4[C]+3[H]$ of the ample cone to the rays $2a[C] + b[H]$, where $a,b \in \mathbb{Z}$, $b > 0$ are solutions to the Pell equation $b^2 - 2a^2 = 1$. As the number of iterations tends to infinity, $\frac{b}{a}$ tends to $\pm\sqrt{2}$, so the limiting rays are $\pm\sqrt{2}[C] + [H]$ as expected.
\end{example}

\begin{remark}
It can happen that the cone $\Amp(S)$ is not rational polyhedral. However, it can also be shown (see \cite[Lemma 2.4]{frak3s}) that $\Amp(S)$ always has a rational polyhedral fundamental domain under the action of $G_S$. Therefore if $\Amp(S)$ is not rational polyhedral, then $G_S$ must be infinite and results in Sect. \ref{section:automorphisms} can be used to show that the automorphism group of $S$ must also be infinite. This is discussed further in the section on the cone conjecture in Huybrechts' lecture notes \cite[Sect. 8.4]{lok3s}.
\end{remark}

\subsection{Genus One Fibrations and the Ample Cone}

One application of this material is to detect genus one fibrations on a projective K3 surface $S$. Note here that we do not assume the existence of a section, so we cannot use the results from Sect. \ref{sect:Hpol} about the existence of an $H$-polarization.

We begin by noting that, if $S$ has a genus one fibration $\pi\colon S \rightarrow B$, then the class $[E]$ of a fibre corresponds to a class in $\NS(S)$ which is in the boundary of the closure $\overline{\Amp(S)}$. To see this, first, note that $\langle [E],[E]\rangle = 0$. Furthermore, if $[C]$ is the class of an irreducible curve on $S$, then either $C$ is a curve in a fibre of $\pi$, in which case $\langle [E],[C]\rangle =0$, or $\pi|_C$ is a surjective map of curves, in which case $\langle [E], [C]\rangle = \deg \pi|_C > 0$. Thus by the Nakai-Moishezon criterion, $[E]$ lies in $\overline{\Amp(S)}$ but not in $\Amp(S)$.

In fact, the converse of this statement is also true:
\begin{theorem}\textup{\cite[Thm. 3.1]{ttastk3}} Let $S$ be a projective K3 surface and let $D$ be a class in the closure of $\Amp(S)$ that has $\langle D,D\rangle = 0$. Then there exists an $n$ such that $nD$ is the class of a fibre in a genus one fibration on $S$. 
\end{theorem}

Now let $D$ be \emph{any} nonzero class in $\NS(S)$ with $\left<D,D \right> = 0$. Then, possibly after negating, we may assume that $D$ is in the closure of the positive cone $\calC_S$. Since $\Amp(S)$ is a fundamental domain for the action of $W_S$ on $\calC_S$, there must be some $\gamma \in W_S$ such that $\gamma(D)$ is in $\overline{\Amp(S)}$. Thus $\gamma(D)$ is the class of a fibre in a genus one fibration. Thus we have:

\begin{corollary}\textup{\cite[Cor. 3.3]{ttastk3}} A projective K3 surface $S$ admits a genus one fibration if and only if the lattice $\NS(S)$ admits a nonzero element $u$ with $\langle u, u \rangle = 0$.\end{corollary}

\begin{example}[Anticanonical hypersurfaces in $\mathbb{P}^1 \times \mathbb{P}^2$] \label{ex:genusone}
Look at the anticanonical K3 surfaces in $\mathbb{P}^1 \times \mathbb{P}^2$ given in Ex. \ref{example:P1P2}. Generically, such K3 surfaces have N\'eron-Severi lattice 
\[\left( \begin{matrix} 0 & 3 \\ 3 & 2 \end{matrix} \right).\]
Hence $\NS(S)$ admits an element of square 0 and thus $S$ admits a genus one fibration. One may exhibit this fibration by restricting the natural projection $\mathbb{P}^1 \times \mathbb{P}^2 \to \mathbb{P}^1$ to the surface $S$. Note that $H$ is not a sublattice of $\NS(S)$, so this \emph{cannot} be an elliptic fibration.
\end{example}

Since every indefinite lattice of rank $n \geq 5$ contains an element of square $0$ \cite[Cor. IV.3.2]{aca} we have:
\begin{corollary}
 Let $S$ be a projective K3 surface with $\rank \NS(S) \geq 5$. Then $S$ admits a genus one fibration.
\end{corollary}

\subsection{Automorphisms of K3 surfaces}
\label{section:automorphisms}

A second application of this material is to the study of the automorphisms of a projective K3 surface $S$. We can produce such automorphisms using the Strong Torelli Theorem (Thm. \ref{strongtorelli}) along with some lattice theory. 

Begin by letting $f_0$ be an automorphism of $\NS(S)$. Then, by \cite[Cor. 1.5.2]{isbfa}, the automorphism
\[f_0 \oplus \Id\colon \NS(S) \oplus \T(S) \rightarrow \NS(S) \oplus \T(S)\]
extends uniquely to an automorphism $f$ of the lattice $\Lambda_{\mathrm{K3}}$ if and only if $f_0$ acts trivially on the discriminant lattice $A_{\NS(S)}$ of $\NS(S)$ (recall here that $\T(S)$ denotes the transcendental lattice of $S$, see Sect. \ref{section:hodge}). 

Assume that this is the case. By definition, the holomorphic $2$-form $\sigma \in H^{2,0}(S)$ sits inside $\T(S)$, so $f$ fixes the period point of $S$. Therefore $f$ is induced by a non-trivial automorphism of $H^2(X,\mathbb{Z})$ which preserves the period $\sigma$. 

If we further assume that $f_0$ is contained in the group $G_S$ of automorphisms of $\NS(S)$ which preserve the ample cone $\Amp(S)$, then $f$ sends some ample class of $S$ to another ample class on $S$. Thus the Strong Torelli Theorem (Thm. \ref{strongtorelli}) tells us that $f$ induces a unique isomorphism on $S$. We have:

\begin{proposition}\textup{\cite[Sect. 7]{ttastk3}}
Let $S$ be a projective K3 surface. The subgroup of $\Aut(S)$ which fixes $\T(S)$ is isomorphic to the finite index subgroup of $G_S$ that acts trivially on the discriminant lattice $A_{\NS(S)}$.
\end{proposition}

Such automorphisms are called \emph{symplectic automorphisms}, since they preserve the holomorphic symplectic form $\sigma$ of $S$. We will denote the group of such automorphisms by $\Aut(S)^s$. There is an embedding $i\colon \Aut(S)^{s} \hookrightarrow \Aut(S)$.

Another important theorem which follows with minimal effort from lattice theory and the Strong Torelli Theorem is:

\begin{theorem}\textup{\cite[Sect. 7]{ttastk3}\cite[Cor. 4.2.4]{iqfaag}}
\label{theorem:autfinite}
The cokernel of the embedding $i$ is finite for all projective K3 surfaces $S$. Thus the group $\Aut(S)$ is finite if and only if $G_S$ is finite.
\end{theorem}

The subgroup of $\Aut(S)$ which does not fix $\T(S)$ tends to be quite small, and is called the \emph{group of non-symplectic automorphisms of $S$}. There has been much work done towards classification of finite groups of symplectic and non-symplectic automorphisms on K3 surfaces. Nikulin provided a classification of cyclic symplectic automorphisms of K3 surfaces in \cite{fagkk3s} and Mukai completed the classification of symplectic automorphism groups of K3 surfaces in \cite{fgaks3mg} (see also \cite{gcbk3s}). Non-symplectic automorphism groups of K3 surfaces were also classified by Nikulin in \cite{fagkk3s}. Recently, more work has been done towards explicitly exhibiting and classifying K3 surfaces admitting non-symplectic automorphisms; see, for example, \cite{k3snsapo}.

\begin{remark}
Finiteness of the group $G_S$ is equivalent to the index of $W_S$ being finite in $\mathrm{O}(\NS(S))$ (since $\mathrm{O}^+(\NS(S))$ has finite index in $\mathrm{O}(\NS(S))$). Lattices whose Weyl groups are of finite index in their orthogonal groups are called \emph{reflexive}. Nikulin \cite{fggahfrsg2r} \cite{stk3fgapgrt} has produced a classification of reflexive hyperbolic lattices, which in turn gives a classification of lattice polarized K3 surfaces with finite automorphism group \cite{ak3sfag}.
\end{remark}

\begin{example}[K3 surfaces with infinitely many symplectic automorphisms]
Let $L$ be a lattice of rank $2$ with Gram matrix
\[\left(\begin{matrix} 2na & nb \\ nb & 2nc \end{matrix} \right),\]
where $n > 1$ and $4ac - b^2 < 0$ is not a perfect square. Then $L$ admits no classes $\delta$ of square $(-2)$, so if a K3 surface $S$ has $\NS(S) \cong L$, then $W_S$ is trivial and $\calC_S = \Amp(S)$.  Therefore, the finite index subgroup of $\mathrm{O}^+(\NS(S))$ fixing the discriminant $A_{\NS(S)}$ is isomorphic to the group of symplectic automorphisms of $S$. This group is infinite cyclic and closely related to the group of units of a subring of $\mathbb{Q}\left(\sqrt{4ac - b^2}\right)$, we refer the interested reader to \cite{aibqfk3spn2} for details.
\end{example}

\subsection{The K\"{a}hler Cone}

Now suppose that $S$ is a smooth K\"{a}hler surface (that is not necessarily projective). Then we have:

\begin{definition} The \emph{K\"ahler cone} of $S$ is the open convex cone $\K(S)$ of all K\"ahler classes in $H^{1,1}(S,\R)$.
\end{definition}

The K\"ahler cone can be constructed in a very similar way to the ample cone on a projective surface (in fact, as we shall see, the two are very closely related). Begin by considering the set
\[\{u \in H^{1,1}(S,\R) \mid \langle u,u \rangle > 0\},\]
which is the analogue of $\NS(S)^+$ from Sect. \ref{sect:amplecone}. This set consists of two disjoint cones. All of the K\"ahler classes belong to one of them, the \emph{positive cone}, which will again be denoted $\calC_S$ (this seems like a confusing choice of terminology but, when $S$ is projective, the positive cone from Sect. \ref{sect:amplecone} is simply the intersection of $\calC_S$ with $\NS(S)$). Then we find:

\begin{theorem}
\label{theorem:Kahler2} \textup{\cite[Cor. VIII.3.9]{bpv}}
Let $S$ be a K3 surface and let $\Delta^+(S)$ be the set of classes in $\NS(S)$ which are represented by smooth rational curves on $S$. Then the K\"ahler cone of $S$ is given by the intersection between the positive cone $\calC_S$ and the set
\[\{u \in H^{1,1}(S,\R) \mid \langle u,\delta \rangle > 0\  \mathrm{for}\  \mathrm{all}\ \delta \in \Delta^+(S) \}.\]
\end{theorem}

From this and Thm. \ref{theorem:Kahler}, we see that if $S$ is projective then $\Amp(S) = \K(S) \cap \NS(S) \otimes \R$. Given this, most of the results on the ample cone that we saw in Sect. \ref{sect:amplecone} also hold for the K\"ahler cone. In particular, the action of the group $W_S$ (defined as before) extends to all of $H^2(S,\Z)$ and this action preserves the positive cone $\calC_S$. The closure of the K\"ahler cone $K(S)$ is then a fundamental domain for the action of $W_S$ on $\calC_S$ \cite[Prop. VIII.3.10]{bpv}.

\begin{remark} It is easy to see that these results on the K\"ahler cone imply the corresponding results about the ample cone, so many references choose to focus on the K\"ahler cone first. However, we find the ample cone to be a conceptually simpler object to study, so we decided to reverse the order  (in particular, the K\"ahler cone is always $20$ dimensional, whereas the dimension of $\Amp(S)$ depends upon $\NS(S)$, so we can find examples where $\Amp(S)$ is small enough to write down explicitly).
\end{remark}

\section{Further Reading}

For the interested reader, more detailed information about K3 surfaces and the period map, including proofs of the Torelli Theorems, may be found in Chap. VIII of the book by Barth, Hulek, Peters and van de Ven \cite{bpv}. For a more in-depth discussion of the construction of moduli spaces of K3 surfaces and their compactifications, we recommend the article by Gritsenko, Hulek and Sankaran \cite{mk3sism}. Further information about the moduli space of polarized K3 surfaces and its compactifications (especially the Baily-Borel compactification) may also be found in the book by Scattone \cite{cmsak3s}. An excellent overview of the theory of degenerations may be found in the survey paper by Friedman and Morrison \cite{bgd1}. The best reference for the theory of lattice polarized K3 surfaces and their moduli is still probably Dolgachev's original paper \cite{mslpk3s}. For more information on the theory of elliptic surfaces, Miranda's book \cite{btes} is an excellent reference. Finally, readers interested in learning more about the ample and K\"{a}hler cones of K3 surfaces can consult the chapter on the ample cone and K\"{a}hler cone in Huybrechts' lecture notes \cite[Chap. 8]{lok3s}.
\medskip

\noindent {\small \textbf{Acknowledgements.} A part of these notes were written while A. Thompson was in residence at the Fields Institute Thematic Program on Calabi-Yau Varieties: Arithmetic, Geometry and Physics; he would like to thank the Fields Institute for their support and hospitality.}

\section*{Appendix: Lattice Theory}
\addcontentsline{toc}{section}{Appendix: Lattice Theory}

In this appendix we present a short description of the lattice theory that is used in the preceding article. The main reference for this section will be \cite{isbfa}.

In this article, we use the word \emph{lattice} in the following sense.

\begin{definition}
A lattice is a pair $(L,\langle\cdot,\cdot\rangle)$ consisting of a finitely generated free $\mathbb{Z}$-module $L$ and an integral symmetric bilinear form $\langle \cdot,\cdot \rangle$ on $L$.
\end{definition}

Often we will suppress the bilinear form $\langle \cdot, \cdot \rangle$ and refer to a lattice simply as $L$. A lattice $L$ is called \emph{non-degenerate} if the $\R$-linear extension of the bilinear form $\langle \cdot, \cdot \rangle$ to the $\R$-vector space $L \otimes_{\Z} \R$ is non-degenerate. For the remainder of this appendix, we will assume that all lattices are non-degenerate.

A lattice $L$ has \emph{signature} $(m,n)$ if, for some basis $u_1, \dots, u_{m+n}$ of $L \otimes_\mathbb{Z} \mathbb{R}$, we have
\[ \langle u_i, u_j \rangle = \left\{ \begin{array}{cl} 1 & \mathrm{if}\ i = j \in\{ 1, \ldots, m\}, \\
-1 & \mathrm{if}\ i = j \in \{m+1,\ldots,m+n\}, \\
0 & \mathrm{if}\ i \neq j. \end{array} \right.\]
If $L$ is of signature $(m,0)$ we call it \emph{positive definite}, and if it has signature $(0,n)$ we say that it is \emph{negative definite}. If a lattice is neither positive nor negative definite, it is called \emph{indefinite}. If a lattice has signature $(m,1)$ we will call it \emph{hyperbolic}. 

Let $L$ be a lattice and $u_i$ a basis of $L$. Then the \emph{Gram matrix} of $L$ is the matrix of integers $g_{i,j} = \langle u_i,u_j \rangle$ and the \emph{discriminant} of $L$, denoted $\disc (L)$, is the absolute value of the determinant of the Gram matrix. Obviously the Gram matrix depends upon the basis chosen, but the discriminant is independent of basis. 

A lattice is called \emph{even} if for every $u$ in $L$,
\[\langle u,u \rangle \equiv 0 \bmod 2.\]
For instance, a root lattice of ADE type is a positive definite even lattice. When dealing with K3 surfaces, all relevant lattices are even. 

A lattice is called \emph{unimodular} if it has discriminant $1$. Up to isomorphism, there is a single even unimodular rank 2 lattice of signature $(1,1)$, which has Gram matrix for some basis given by
\[\left( \begin{matrix} 0 & 1 \\ 1 & 0 \end{matrix} \right).\]
This lattice is called the \emph{hyperbolic plane} and, depending on the author, is denoted $U$ or $H$. We will denote it by $H$.

Now suppose that $L$ and $M$ are two lattices and that $L$ embeds into $M$. Then $L$ is said to be a \emph{sublattice} of $M$. This embedding is called \emph{primitive} if the quotient $M/L$ is torsion-free. Similarly, an element $u \in M$ is called \emph{primitive} if the sublattice of $M$ generated by $u$ is primitively embedded in $M$.

Given a lattice $L$, we may define a second lattice $L^*$, called the \emph{dual lattice} of $L$, as follows. Consider the tensor product $L \otimes_\mathbb{Z} \mathbb{Q}$, with bilinear form induced by the $\mathbb{Q}$-linear extension of $\langle\cdot , \cdot \rangle$. Then define $L^*$ to be the subgroup of $L \otimes_\mathbb{Z} \mathbb{Q}$ made up of elements $v$ which satisfy $\langle v,u \rangle \in \mathbb{Z}$ for all $u \in L$, equipped with the integral binear form induced by $\langle \cdot,\cdot \rangle$. Note that $L$ is a sublattice of $L^*$.

For even lattices $L$, we may use this to define a more refined version of the discriminant, called the \emph{discriminant lattice} of $L$. This is given by the finite group
\[A_L:=L^*/L.\]
This group is equipped with a quadratic form and a bilinear form as follows: take $u,v \in L^*$ and let $\overline{u},\overline{v}$ be their images in $A_L$, then define
\[q_L(\overline{u}) = \langle u,u \rangle \bmod 2 \mathbb{Z}\]
and 
\[b_L(\overline{u},\overline{v}) = \langle u,v \rangle \bmod \mathbb{Z}.\]
Note that if $u,v \in L$, then the fact that $L$ is an even lattice implies that $q_L(\overline{u}) = 0$ and $b_L(\overline{u},\overline{v}) =0$, so $q_L$ and $b_{L}$ are well-defined. The group $A_L$ is finite and $|A_L| = \disc (L)$. 

The invariant $A_L$ is obviously finer than just the discriminant of the lattice, but its true strength is made evident by the following proposition of Nikulin.

\begin{proposition}\textup{\cite[Cor. 1.13.3]{isbfa}}
\label{proposition:uniqueness}
Let $L$ be an even indefinite lattice of signature $(m,n)$ and rank $m+n$, with discriminant lattice $A_L$. Let $\ell(L)$ denote the minimal number of generators of $A_L$. If $\ell(L) \leq m+n-2$, then any other lattice with the same rank, signature and discriminant lattice is isomorphic to $L$.
\end{proposition}

\subsection*{Overlattices}
\addcontentsline{toc}{subsection}{Overlattices}

Now assume that $L$ and $M$ are two even lattices of the same rank, such that $L$ embeds inside of $M$. Then we say that $M$ is an \emph{overlattice} of $L$. If we begin with a lattice $M$, then it is easy to compute all possible sublattices of maximal rank of $L$, but the problem of computing all possible overlattices of $L$ is more subtle. It is solved by the following theorem:

\begin{theorem} \textup{\cite[Prop. 1.4.1]{isbfa}}
\label{thm:overlattices}
Let $L$ be an even lattice. Then there is a bijection between subgroups $G$ of $A_L$ on which the form $q_L$ satisfies $q_L(u) = 0$ for all $u \in G$ and overlattices $L_G$ of $L$.

Furthermore, the discriminant form of the lattice $L_G$ associated to the subgroup $G$ is given by the form $q_L$ restricted to $G^\perp/G$, where orthogonality is measured with respect to $b_L$.
\end{theorem}

The main practical use of this proposition is to determine when a specific lattice is primitively embedded in another. In particular, if a lattice $L$ has no non-trivial overlattices, then any embedding of $L$ into another lattice $M$ must be primitive. However, it can also be used to explicitly compute the possible overlattices of a given lattice, as illustrated by the next example.

\begin{example}
\label{example:HE7E7}
Let $L$ be the lattice $H \oplus (-E_7) \oplus (-E_7) \oplus (-A_3)$. Then $A_L$ is isomorphic to $\mathbb{Z}/2 \oplus \mathbb{Z}/2 \oplus \mathbb{Z}/4$ with generators $u,v,w$ respectively. It may be checked explicitly that 
\[q_L(u) = q_L(v) = \frac{1}{2}\]
and 
\[q_L(w) = \frac{5}{4},\]
and that $u,v,w$ are mutually orthogonal with respect to $b_L$.

One checks easily that the only nontrivial element $Q$ in $A_L$ with $q_L(Q) = 0$ is $Q = u + v + 2w$, which has order $2$. Thus $L$ has a unique overlattice $L_G$ of index $2$, corresponding to the subgroup $G$ of $A_L$ generated by $Q$. 

One may construct $L_G$ concretely in the following way: let $\hat{Q}$ be some element of $L^*$ whose image in $A_L$ is $Q$, then $L_G$ can be identified as the sublattice of $L\otimes_\mathbb{Z} \mathbb{Q}$ spanned by  $\hat{Q}$ and the image of $L$ in $L\otimes_\mathbb{Z} \mathbb{Q}$.

However, it is often simpler to use Prop. \ref{proposition:uniqueness} to identify the overlattice $L_G$. The subgroup $G$ of $A_L$ generated by $Q$ has orthogonal complement generated by $Q$ and $v + w$. Modulo $G$, this group is cyclic of order four and 
\[q_L(v+ w) = -\frac{1}{4}.\]
Therefore, the overlattice $L_G$ of $L$ associated to $G$ has rank $19$, signature $(1,18)$ and discriminant group of order $4$ with a generator satisfying $q_L(v+w) = -1/4$. 

Now, the lattice $M = H \oplus (-E_8) \oplus (-E_8) \oplus \langle -4 \rangle$ also has rank $19$, signature $(1,18)$ and discriminant group of order $4$ with generator $e$ satisfying $q_M(e) = -1/4$ so, by Prop. \ref{proposition:uniqueness}, the overlattice $L_G$ of $L$ must be isomorphic to the lattice $M$.
\end{example}

\begin{remark}
Note that if we replaced the lattice $H \oplus (-E_7) \oplus (-E_7) \oplus (-A_3)$ with the lattice $(-E_7) \oplus (-E_7) \oplus (-A_3)$ then we could not use Prop. \ref{proposition:uniqueness} here, since the second lattice is not indefinite.
\end{remark}

\bibliography{books}
\bibliographystyle{amsplain}

\end{document}